\documentclass[12pt]{amsart} 
\usepackage[mathscr]{eucal} 
\usepackage{amsmath,amsfonts}
\parskip=\smallskipamount 
\newtheorem{theorem}{Theorem}[section] 
\newtheorem{proposition}[theorem]{Proposition} 
\newtheorem{corollary}[theorem]{Corollary} 
\newtheorem{lemma}[theorem]{Lemma} 
\theoremstyle{definition} 
 
\newtheorem{example}[theorem]{Example} 
\newtheorem{examples}[theorem]{Examples}
 
\newtheorem{remark}[theorem]{Remark} 
\newtheorem{remarks}[theorem]{Remarks}
\makeatletter 
\@addtoreset{equation}{section} 
\makeatother

\newcommand{\CC}{{\mathbb C}} 
\newcommand{\NN}{{\mathbb N}}

\newcommand{\cA}{{\mathcal A}} 
\newcommand{\cB}{{\mathcal B}} 
\newcommand{\cC}{{\mathcal C}} 
 
\newcommand{\cE}{{\mathcal E}} 
\newcommand{\cF}{{\mathcal F}} 
\newcommand{\cG}{{\mathcal G}} 
\newcommand{\cH}{{\mathcal H}}

\newcommand{\cK}{{\mathcal K}} 
\newcommand{\cL}{{\mathcal L}}

\newcommand{\cR}{{\mathcal R}} 
\newcommand{\cS}{{\mathcal S}}

\newcommand{\cV}{{\mathcal V}} 
 
\newcommand{\cZ}{{\mathcal Z}}

\newcommand{\fk}{\mathbf{k}}
\newcommand{\fl}{\mathbf{l}}

\newcommand{\Ra}{\Rightarrow}

\newcommand{\ra}{\rightarrow} 
\newcommand{\ol}{\overline}

\newcommand{\tr}{\operatorname{tr}} 
\let\phi=\varphi 
\newcommand{\iac}{\mathrm{i}}

\newcommand{\lin}{\operatorname{Lin}}
\newcommand{\supp}{\operatorname{supp}}

\newcommand{\nr}[1]{\vspace{0.1ex}\noindent\hspace*{12mm}\llap{\textup{(#1)}}} 
 
\begin{document} 
\title[Invariant Kernels with Values Adjointable Operators I.]{Representations 
of $*$-semigroups  associated to invariant kernels with values adjointable 
operators. I}\thanks{The 
second named author's work supported by a grant of the Romanian 
National Authority for Scientific Research, CNCS  UEFISCDI, project number
PN-II-ID-PCE-2011-3-0119.}
 
 \date{\today}
 
 \author[S. Ay]{Serdar Ay}
 \address{Department of Mathematics, Bilkent University, 06800 Bilkent, Ankara, 
Turkey}
 \email{serdar@fen.bilkent.edu.tr}
 
\author[A. Gheondea]{Aurelian Gheondea} 
\address{Department of Mathematics, Bilkent University, 06800 Bilkent, Ankara, 
Turkey, \emph{and} Institutul de Matematic\u a al Academiei Rom\^ane, C.P.\ 
1-764, 014700 Bucure\c sti, Rom\^ania} 
\email{aurelian@fen.bilkent.edu.tr \textrm{and} A.Gheondea@imar.ro} 

\begin{abstract} We consider positive semidefinite 
kernels valued in the $*$-algebra of adjointable operators on a VE-space
(Vector Euclidean space) and that are
invariant under actions of $*$-semigroups. 
A rather general dilation theorem is stated and proved: for these kind of 
kernels, representations of the $*$-semigroup on either the
VE-spaces of linearisation of the kernels or on their reproducing kernel
VE-spaces are obtainable. 
We point out the reproducing kernel fabric of dilation theory and
we show that the general theorem unifies many dilation results at the 
non topological level.
\end{abstract} 

\maketitle 

\section*{Introduction}

Starting with the celebrated Naimark's dilation theorems in \cite{Naimark1}
and \cite{Naimark2}, a powerful dilation theory for operator  valued maps
was obtained through 
results of B.~Sz.-Nagy \cite{BSzNagy}, W.F.~Stinespring 
\cite{Stinespring}, and their generalisations to VH-spaces 
(Vector Hilbert spaces) by R.M.~Loynes \cite{Loynes1}, 
or to Hilbert $C^*$-modules by G.G.~Kasparov \cite{Kasparov}. 
Taking into account the diversity of dilation theorems for operator valued
maps, there is a natural question, whether one can
unify all, or the most, of these dilation theorems, under one theorem. An
attempt to approach this question was made in \cite{Gheondea} by using
the notion of VH-space over an admissible space, introduced by 
R.M.~Loynes \cite{Loynes1}, \cite{Loynes2}. 
As a second step in this enterprise, an investigation at the "ground level", 
that is, a non topological approach, makes perfectly sense. In addition,
an impetus to pursue this way was given to us by the recent investigation on 
closely related problems, e.g.\ non topological theory 
for operator spaces and operator systems, cf.\ \cite{PaulsenTomforde}, 
\cite{EsslamzadehTaleghani}, \cite{PaulsenTodorovTomforde}, 
\cite{EsslamzadehTurowska}.

The aim of this article is to present a general non topological approach 
to dilation theory based on positive semidefinite kernels that are 
invariant under actions of $*$-semigroups and with values adjointable 
operators on VE-spaces (Vector Euclidean spaces) over ordered $*$-spaces.
More precisely, we show that at the level of conjunction of order with
$*$-spaces or $*$-algebras and operator valued maps, 
one can obtain a reasonable dilation theory that 
contains the fabric of most of the more or less topological 
versions of classical dilation theorems. 
In addition, we integrate into non topological dilation theory, on equal foot,
the reproducing kernel technique and show that almost each dilation theorem 
is equivalent to a realisation as a reproducing kernel space with additional
properties. Our approach is based on ideas already 
present under different topological versions of dilation theorems in 
\cite{ParthasarathySchmidt}, \cite{EvansLewis}, \cite{Loynes1}
\cite{ConstantinescuGheondea1}, \cite{ConstantinescuGheondea2}, 
\cite{Murphy}, \cite{GasparGaspar}, 
\cite{GasparGaspar2}, \cite{GasparGaspar1}, \cite{Szafraniec}, 
\cite{Heo}, \cite{Gheondea} and, probably, many others.

We briefly describe the contents of this article. In Section~\ref{s:p} 
we fix some terminology
and facts on ordered $*$-spaces, ordered $*$-algebras, VE-spaces over 
ordered $*$-spaces, and VE-modules over ordered $*$-algebras. On these
basic objects, one can build the ordered $*$-algebras of adjointable operators 
on VE-spaces or VE-modules. We provide many examples that illustrate the
richness of this theory, even at the non topological level.

Then, in Section~\ref{s:lik}, we consider the main object of investigation which
refers to positive semidefinite kernels with values adjointable operators 
on VE-spaces. We make a preparation by showing that, although analogs of 
Schwarz Inequality are missing at this level of generality, 
some basic results can be obtained by different techniques.
In order to achieve a sufficient generality that allows to recover known 
dilation theorems for both $*$-semigroups (B.~Sz.-Nagy) and $*$-algebras 
(Stinespring), in view of \cite{Gheondea}, we consider positive 
semidefinite kernels that are invariant under actions of 
$*$-semigroups and that have values adjointable operators on VE-spaces. In
Lemma~\ref{l:twopos} we show that, for a $2$-positive kernel, 
if boundedness in the sense of
Loynes is assumed for all the operators on the diagonal, then the entire kernel 
is made up by bounded operators. In this way we explain how the investigation
of this article is situated with respect to that in \cite{Gheondea}.
We briefly show the connection between linearisations and reproducing kernel
spaces at this level of generality.
It is this stage when we are able to state and prove the main result, 
Theorem~\ref{t:vhinvkolmo} 
that, basically, shows that this kind of kernels produce $*$-representations
on "dilated" VE-spaces that linearise the kernel or, equivalently, on 
reproducing kernel VE-spaces that can be explicitly described.

Finally, in Section~\ref{s:rtu}, as consequences of 
Theorem~\ref{t:vhinvkolmo}, we show how non topological
versions of most of the known dilation theorems \cite{BSzNagy}, 
\cite{Stinespring}, \cite{Loynes1}, \cite{Kasparov}, \cite{Joita}
can be obtained. 
On the other hand, in order to unify the known dilation theorems in 
topological versions, one needs certain topological structures on 
ordered $*$-spaces 
and VE-spaces, that lead closely to the VH-spaces over admissible spaces, 
as considered in \cite{Loynes1}. This way was followed, to a certain extent, 
in \cite{Gheondea} but, in order to obtain a sufficiently large generality 
allowing to cover most of the known topological dilation theory, 
one needs more flexibility by moving beyond bounded operators in the sense of 
Loynes. 
We will consider this in subsequent articles.
 
\section{Preliminaries}\label{s:p}

In this section we briefly review most of the definitions and some basic
facts on VE-spaces over ordered $*$-spaces, inspired by cf.\
R.M. Loynes, \cite{Loynes1}, \cite{Loynes2}, and \cite{Loynes3}. 
We slightly modify some definitions in order to match the requirements of 
this investigation, notably by separating the non topological from the 
topological cases and by adhering to the convention, that is very popular
in Hilbert $C^*$-modules, to let gramians be linear in the second variable and
conjugate linear in the first variable, for consistency.

\subsection{Ordered $*$-Spaces.}\label{ss:oss}
A complex vector space $Z$ is called \emph{ordered $*$-space},
c.f.\ \cite{PaulsenTomforde}, if:
\begin{itemize}
\item[(a1)] $Z$ has an \emph{involution} $*$, that is, a map 
$Z\ni z\mapsto z^*\in Z$ 
that is \emph{conjugate linear} 
(($s x+t y)^*=\ol s x^*+\ol t y^*$ for all 
$s,t\in\CC$ and all $x,y\in Z$) and \emph{involutive} 
($(z^*)^*=z$ for all $z\in Z$). 
\item[(a2)] In $Z$ there is a \emph{cone} $Z^+$ ($s x+t y\in Z^+$ 
for all numbers $s,t\geq 0$ and all $x,y\in Z^+$), that is  
\emph{strict} ($Z^+\cap -Z^+=\{0\}$), and consisting of 
\emph{self-adjoint elements} 
only  ($z^*=z$ for all 
$z\in Z^+$). This cone is used to define a \emph{partial order} on the real 
vector space of all selfadjoint elements in $Z$: 
$z_1  \geq z_2$ if $z_1-z_2 \in Z^+$.
\end{itemize}

Recall that a \emph{$*$-algebra} $\cA$ is a complex algebra onto which 
there is defined an \emph{involution} $\cA\ni a\mapsto a^*\in\cA$, that is, 
$(\lambda a+\mu b)^*=\overline \lambda a^*+\overline \mu b^*$, 
$(ab)^*=b^*a^*$, and $(a^*)^*=a$, for all $a,b\in\cA$ and all 
$\lambda,\mu\in\CC$. 

An \emph{ordered $*$-algebra} $\cA$ is a $*$-algebra 
such that it is an ordered $*$-space, more precisely, it has the following
property.
\begin{itemize}
\item[(osa1)] There exists a strict cone $\cA^+$ in $\cA$ such that for any 
$a\in\cA^+$ we have $a=a^*$.
\end{itemize}
Clearly, any ordered $*$-algebra is an ordered $*$-space. In particular, 
given $a\in\cA$, we denote $a\geq 0$ if $a\in\cA^+$ and, for 
$a=a^*\in\cA$ and $b=b^*\in\cA$, we denote $a\geq b$ if $a-b\geq 0$.

\begin{remark}\label{r:isa} In analogy with the case of $C^*$-algebras, given 
a $*$-algebra $\cA$, one defines an element $a\in\cA$ to be 
\emph{$*$-positive} 
if $a=\sum_{k=1}^n a_k^*a_k$ for some natural number 
$n$ and some elements $a_1,\ldots,a_n\in \cA$. The collection of 
all $*$-positive elements in a $*$-algebra is 
a cone, but it may fail to be strict and hence, associated is only a 
\emph{quasi-order}, e.g.\ see \cite{EsslamzadehTaleghani} for a recent 
investigation. Thus, our definition of an ordered $*$-algebra 
specifies a strict cone $\cA^+$ and, in general, it does not refer to the cone of 
$*$-positive elements as defined above, except special cases as, for
example, pre $C^*$-algebras or pre locally $C^*$-algebras. 
\end{remark}

\begin{examples}\label{ex:eos}
(1) Any $C^*$-algebra, e.g.\ see \cite{Arveson}, 
$\cA$ is an ordered $*$-algebra and
any $*$-subspace $\cS$ of a $C^*$-algebra $\cA$, with the 
positive cone $\cS^+=\cA^+\cap\cS$ and all other operations (addition, 
multiplication with scalars, and involution) inherited from $\cA$, is a 
$*$-space.

(2) Any pre-$C^*$-algebra is an ordered $*$-algebra. Any 
$*$-subspace $\cS$ of a pre-$C^*$-algebra $\cA$ is an ordered 
$*$-space, with the positive cone $\cS^+=\cA^+\cap\cS$ and all other 
operations inherited from $\cA$.

(3) Any locally $C^*$-algebra, see \cite{Inoue}, \cite{Phillips}, 
is an ordered $*$-algebra. 
In particular, any $*$-subspace $\cS$ of a locally $C^*$-algebra $\cA$, 
with the cone $\cS^+=\cA^+\cap\cS$ and all other operations inherited from 
$\cA$, is an ordered $*$-space. 

(4) Any locally pre-$C^*$-algebra is an ordered $*$-algebra. Any
$*$-subspace $\cS$ of a locally pre-$C^*$-algebra is an ordered
$*$-space, with $\cS^+=\cA^+\cap\cS$ and all other operations inherited from 
$\cA$.

(5) Let $\cH$ be an infinite dimensional separable Hilbert space and let
$\cC_1$ be 
the trace-class ideal, that is, the collection of all linear bounded operators
$A$ 
on $\cH$ such that $\tr(|A|)<\infty$. $\cC_1$ is a $*$-ideal of $\cB(\cH)$. 
Positive elements in $\cC_1$ are defined in the sense 
of positivity in $\cB(\cH)$. $\cC_1$ is an ordered $*$-space.

(6) Let $V$ be a complex vector space and let
$V^\prime$ be its conjugate dual space. On the vector space 
$\mathcal{L}(V,V^\prime)$ of all linear
operators $T\colon V\rightarrow V^\prime$, a natural notion of positive 
operator can be defined: $T$ is \emph{positive} if $(Tv)(v)\geq 0$ for all 
$v\in V$. Let
$\mathcal{L}^+(V,V^\prime)$ be the collection of all positive operators  and 
note that it is a strict cone. The involution $*$ 
in $\mathcal{L}(V,V^\prime)$ is defined in the following way: for any
$T\in\mathcal{L}(V,V^\prime)$, $T^*=T^\prime|V$, that is, the restriction to 
$V$ of the dual operator 
$T^\prime\colon V^{\prime\prime}\rightarrow V^\prime$. With respect to 
the cone $\mathcal{L}^+(V,V^\prime)$ and the 
involution $*$ just defined, $\mathcal{L}(V,V^\prime)$ becomes an ordered 
$*$-space. 
See A.~Weron \cite{Weron}, as well as D.~Ga\c spar and P.~Ga\c spar 
\cite{GasparGaspar}.

(7) Let $X$ be a nonempty set and denote by $\cK(X)$ the collection of 
all complex valued kernels on $X$, that is, $\cK(X)=\{k\mid k\colon X\times 
X\rightarrow \CC\}$, considered as a complex vector space with the 
operations of addition and multiplication of scalars defined elementwise. 
An involution 
$*$ can be defined on $\cK(X)$ as follows: $k^*(x,y)=\overline{k(y,x)}$, for all 
$x,y\in X$ and all $k\in\cK(X)$. The cone $\cK(X)^+$ consists on all 
\emph{positive semidefinite} kernels, that is, those kernels $k\in\cK(X)$ with 
the property that, for any $n\in\NN$ and any $x_1,\ldots,x_n\in X$, the 
complex matrix $[k(x_i,x_j)]_{i,j=1}^n$ is positive semidefinite. 

On $\cK(X)$ a multiplication can be defined elementwise: if $k,l\in\cK(X)$
then $(kl)(x,y)=k(x,y)l(x,y)$ for all $x,y\in X$. With respect to this multiplication
and the other operations described before, $\cK(X)$
is an ordered $*$-algebra.

Using the notion of \emph{Schur product}, e.g.\ see \cite{Paulsen},
it can be proven that the ordered 
$*$-algebra $\cK(X)$ has the following property: if $k,l\in\cK(X)$ are positive
semidefinite kernels, then $kl$ is positive semidefinite. However, this is a 
case that illustrates Remark~\ref{r:isa}: it is not true, in general, that kernels
of type $k^*k$ are positive semidefinite.

(8) Let $\cA$ and $\cB$ be two ordered $*$-spaces. In addition, we assume 
that the specified strict cone $\cA^+$ linearly generates $\cA$. On 
$\cL(\cA,\cB)$, the vector space of all linear maps $\phi\colon \cA\ra\cB$, 
we define an involution: $\phi^*(a)=\phi(a^*)^*$, for all $a\in\cA$. A linear 
map $\phi\in\cL(\cA,\cB)$ is called positive if $\phi(\cA^+)\subseteq \cB^+$. It
is easy to see that $\cL(\cA,\cB)^+$, the collection of all positive maps from 
$\cL(\cA,\cB)$, is a cone, and that it is strict because $\cA^+$ 
linearly generates $\cA$. In addition, any $\phi\in\cL(\cA,\cB)^+$ 
is selfadjoint, 
again due to the fact that $\cA^+$ linearly generates $\cA$. Consequently,
$\cL(\cA,\cB)$ has a natural structure of ordered $*$-space.

(9) Let $\{Z_\alpha\}_{\alpha\in A}$ be a family of ordered $*$-spaces such that,
for each $\alpha\in A$, $Z_\alpha^+$ is the specified strict cone of positive
elements in $Z_\alpha$. 
On the product space $Z=\prod_{\alpha\in A}Z_\alpha$ let 
$Z^+=\prod_{\alpha\in A}Z_\alpha^+$ and observe that $Z^+$ is a
strict cone. Letting the involution $*$ on $Z$ be defined elementwise, it
follows that $Z^+$ consists on selfadjoint elements only. In this way, $Z$
is an ordered $*$-space.

(10) Let $Z$ be an ordered $*$-space with the specified strict cone $Z^+$. 
A subspace $J$ of $Z$ is called an
\emph{order ideal} if it is \emph{selfadjoint}, that is, 
$J=J^*=\{z^*\mid z\in J\}$, and \emph{solid}, that is, for any 
$z\in J\cap Z^+$ and any $y\in Z^+$
such that $y\leq z$ it follows $y\in J$. Then, on the quotient vector space 
$Z/J$, an
involution $*$ can be defined by: $(z+J)^*=z^*+J$, for $z\in Z$. Also, letting
$(Z/J)^+=\{z+J\mid z\in Z^+\}$, it follows that $(Z/J)^+$ is a strict cone in
$Z/J$ consisting of selfadjoint elements only and, hence, $Z/J$ is an ordered
$*$-space. See \cite{PaulsenTomforde}.
\end{examples}

\subsection{Vector Euclidean Spaces and Their Linear Operators.}\label{ss:ves}
Given a complex linear space $\cE$ and an
ordered $*$-space space $Z$, a \emph{$Z$-gramian}, also called a 
\emph{$Z$-valued inner product}, is a mapping  
$\cE\times \cE\ni (x,y) \mapsto [x,y]\in Z$ subject to 
the following properties:
\begin{itemize}
\item[(ve1)] $[x,x] \geq 0$ for all $x\in \cE$, and $[x,x]=0$ if and only if 
$x=0$.
\item[(ve2)] $[x,y]=[y,x]^*$ for all $x,y\in\cE$.
\item[(ve3)] $[x,\alpha y_1+\beta y_2]=\alpha 
[x,y_1]+\beta [x,y_2]$ for all $\alpha,\beta\in \mathbb{C}$ and 
all $x_1,x_2\in \cE$.
\end{itemize}

A complex linear space $\cE$ onto which a $Z$-gramian 
$[\cdot,\cdot]$ is specified, for a 
certain ordered $*$-space $Z$, is called a \emph{VE-space} 
(Vector Euclidean space) over $Z$, cf.\ \cite{Loynes1}. 

\begin{remark}\label{r:polar}
In any VE-space $\cE$ over an ordered $*$-space $Z$, the familiar 
\emph{polarisation formula}
\begin{equation}\label{e:polar} 4[x,y]=\sum_{k=0}^3 \iac^k 
[(x+\iac^k y,x+\iac^k y],\quad x,y\in \cE,
\end{equation} holds, which shows that the $Z$-valued inner product is 
perfectly defined by the $Z$-valued quadratic map 
$\cE\ni x\mapsto [x,x]\in Z$.

Actually, the formula \eqref{e:polar} is more general: given a pairing 
$[\cdot,\cdot]\colon \cE\times \cE\ra Z$, where $\cE$ is some vector space
and $Z$ is a $*$-space, and assuming that $[\cdot,\cdot]$ satisfies only the
axioms (ve2) and (ve3), then \eqref{e:polar} is still valid.
\end{remark}

The concept of \emph{VE-spaces isomorphism} is also naturally defined: 
this is just a linear bijection $U\colon \cE\ra \cF$, for two VE-spaces 
over the same ordered $*$-space $Z$, which is \emph{isometric}, that is, 
$[Ux,Uy]_\cF=[x,y]_\cE$ for all $x,y\in \cE$.

In general VE-spaces, an analog of the Schwarz Inequality may not hold but 
some of its consequences can be proven using slightly different 
techniques. One such method is provided by the following lemma.

\begin{lemma}[Loynes \cite{Loynes1}]\label{l:sesqui} 
Let $Z$ be an ordered $*$-space, $\cE$ a complex vector space and
$[\cdot,\cdot]\colon \cE\times \cE \rightarrow Z$ a positive semidefinite 
sesquilinear map, 
that is, $[\cdot,\cdot]$ is linear in the second variable, conjugate linear 
in the first variable, and $[x,x]\geq 0$ for all $x\in \cE$. 
If $f\in \cE$ is such that $[f,f]=0$, then $[f,f']=[f',f]=0$ for all 
$f' \in \cE$.
\end{lemma}

Given two VE-spaces $\cE$ and $\cF$, over the same ordered $*$-space 
$Z$, one can consider the vector space $\cL(\cE,\cF)$ of all
linear operators $T\colon \cE\ra\cF$. 
A linear operator $T\in\cL(\cE,\cF)$ is 
called \emph{adjointable} if there exists $T^*\in\cL(\cF,\cE)$ such that
\begin{equation}\label{e:adj} [Te,f]_\cF=[e,T^*f]_\cE,\quad e\in\cE,\ f\in\cF.
\end{equation} The operator $T^*$, if it exists, is uniquely determined by $T$ 
and called its \emph{adjoint}.
Since an analog of the Riesz Representation Theorem for VE-spaces
may not exist, in general, 
there may be not so many adjointable operators. Denote by 
$\cL^*(\cE,\cF)$ the vector space of all adjointable operators from 
$\cL(\cE,\cF)$.
Note that $\cL^*(\cE)=\cL^*(\cE,\cE)$ 
is a $*$-algebra with respect to the involution $*$ 
determined by the operation of taking the adjoint. 

An operator $A\in\cL(\cE)$ is called \emph{selfadjoint} if
$ [Ae,f]=[e,Af]$, for all $e,f\in \cE$.
Clearly, any selfadjoint operator $A$ is adjointable and 
$A=A^*$.
By the polarisation formula \eqref{e:polar}, $A$ is selfadjoint if and only if
$ [Ae,e]=[e,Ae]$, $ e\in\cE$.
An operator $A\in\cL(\cE)$ is \emph{positive} if
\begin{equation}\label{e:pos} [Ae,e]\geq 0,\quad  e\in\cE.\end{equation}
Since the cone $Z^+$ consists of selfadjoint elements only, 
any positive operator is selfadjoint and hence adjointable. 
On the other hand, note that any VE-space isomorphism
is adjointable and hence, it makes sense to call it \emph{unitary}.

\begin{examples}\label{ex:les}
(1) If $\cE$ is some VE-space over an ordered $*$-space $Z$, then $\cL^*(\cE)$
is an ordered $*$-algebra, where the cone of positive elements is defined 
by \eqref{e:pos}. Note that this cone is strict.
In connection with Remark~\ref{r:isa}, note that any operator $A\in\cL^*(\cE)$
that can be represented $A=\sum_{j=1}^N A_j^*A_j$ is positive, but the 
converse, in general, is not true.

(2) With notation as in Example~\ref{ex:eos}.(5), consider $\cC_2$ the ideal of
Hilbert-Schmidt 
operators on $\cH$. Then $[A,B]=A^*B$, for all $A,B\in\cC_2$, is a gramian
with values in the  
ordered $*$-space $\cC_1$ with respect to which $\cC_2$ becomes a
VE-space. More abstract versions of this example have been 
considered by Saworotnow in \cite{Saworotnow}.

(3) Let $\{\cE_\alpha\}_{\alpha\in A}$ be a family of VE-spaces such that, for each
$\alpha\in A$, $\cE_\alpha$ is a VE-space over the ordered $*$-space
$Z_\alpha$. Consider the ordered $*$-space $Z=\prod_{\alpha\in A}Z_\alpha$ as
in Example~\ref{ex:eos}. Consider the vector space $\cE=\prod_{\alpha\in A}
\cE_\alpha$ on which we define
\begin{equation*}[(e_\alpha)_{\alpha\in A},(f_\alpha)_{\alpha\in
      A}]=([e_\alpha,f_\alpha])_{\alpha\in A}\in Z,\quad (e_\alpha)_{\alpha\in A},(f_\alpha)_{\alpha\in
      A}\in \cE. 
\end{equation*} Then $\cE$ is a VE-space over $Z$.

(4)  Given a finite collection of VE-spaces $\cE_1,\dots,\cE_N$, 
over the same ordered $*$-space $Z$, one can define naturally the VE-space
$\cE_1\oplus\cdots\oplus \cE_N$ over $Z$ where, for $e_j,f_j\in\cE_j$, 
$j=1,\ldots,N$ we define
\begin{equation*}[e_1\oplus\cdots\oplus e_N,f_1\oplus\cdots f_N]
=\sum_{j=1}^N [e_j,f_j].
\end{equation*}
We use the notation $\cE^N$ for $\cE_1\oplus\cdots\oplus\cE_N$ 
when $\cE=\cE_j$ for all $j=1,\ldots,N$. Then observe that $\cL^*(\cE^N)$ 
can be naturally identified with $M_N(\cE)$, the space of all $N\times N$ 
matrices with entries in $\cL^*(\cE)$. This identification provides a natural 
structure of ordered $*$-algebra of $\cL^*(\cE^N)$ over $Z$, with an even  
richer structure, see Remarks~\ref{r:cp}.

(5) Let $\cH$ be a pre-Hilbert space having an orthonormal basis and $\cE$ a
VE-space over the ordered $*$-space $Z$. On the algebraic tensor product
$\cH\otimes \cE$ define a gramian by
\begin{equation*} [h\otimes e,l\otimes f]_{\cH\otimes\cE}=\langle
  h,l\rangle_\cH [e,f]_\cE\in Z,\quad
  h,l\in \cH,\ e,f\in \cE,
\end{equation*} and then extend it to $\cH\otimes \cE$ by linearity. By a
standard but rather long argument, e.g.\ similar to \cite{Lance} p.~6, 
it can be proven that, in this way, $\cH\otimes \cE$ becomes a VE-space over 
$Z$ as well.

If $\cH=\CC^n$ for some $n\in\NN$ then, with notation as in item (5), 
it is clear that 
$\CC^n\otimes\cE$ is isomorphic with $\bigoplus_{j=1}^n \cE_j$, with
$\cE_j=\cE$ for all $j=1,\ldots,n$.
\end{examples} 

An operator $A\in\cL(\cE,\cF)$, for two VE-spaces over the same ordered
$*$-space $Z$, is called \emph{bounded in the sense of Loynes}, or simply
\emph{bounded}, if, for some $\mu\geq 0$,
\begin{equation}\label{e:opnorm} [Ah,Ah]_\cF\leq \mu[h,h]_\cE,\quad h\in\cE.
\end{equation}
We denote the class of bounded operators by $\mathcal{B}(\cE,\cF)$. 
For a bounded operator $A\in\cB(\cE,\cF)$, its \emph{operator norm} is denoted
by $\|A\|$ and it is defined by
 square root of the least $\mu\geq 0$ satisfying \eqref{e:opnorm}, that is,
 \begin{equation}\label{e:opnormak}
 \|A\|=\inf\{\sqrt{\mu}\mid \mu\geq 0,\ [Ah,Ah] \leq \mu[h,h],\mbox{ for all } 
h \in \cH\}. 
 \end{equation} It is easy to see that the infimum is actually a 
minimum and hence, that we have
\begin{equation}\label{e:opnorminequa}
 [Ah,Ah] \leq \|A\|^2 [h,h],\quad x\in \cH.
\end{equation}
$\cB(\cE)=\cB(\cE,\cE)$ is a normed algebra with respect to the 
usual algebraic operations and the operator norm, 
cf.\ Theorem 1 in \cite{Loynes2}.

Let $\cB^*(\cE,\cF)$ denote the collection of all bounded and adjointable
linear operators $A\colon\cE\ra\cF$. A \emph{contraction} 
is a linear operator $T\colon\cE\ra\cF$ such that
$[Tx,Tx]\leq [x,x]$ for all $x\in \cH$. 
By Theorem 2 in \cite{Loynes2}, if $T\in\cB^*(\cE,\cF)$ is a contraction then
$T^*$ is a contraction as well, hence, for all $T\in\cB^*(\cE,\cF)$ we have
$\|T^*\|=\|T\|$.

If $A \in \mathcal{B^*}(\cE)$ is selfadjoint, then, by 
Theorem 3 in \cite{Loynes2},
\begin{equation}\label{e:sineg}
-\|A\|[h,h]\leq [Ah,h] \leq \|A\|[h,h],\quad h\in\cE.
\end{equation}
Moreover, if $A$ is a linear operator in $\cE$ and, for some real 
numbers $m,M$, we have
\begin{equation}\label{e:sai}
m[h,h]\leq [Ah,h]\leq M[h,h],\quad h\in \cE,
\end{equation} then $A\in\cB^*(\cE)$ and $A=A^*$. If, in addition, 
$m$ is the minimum and $M$ is the maximum with this property, 
then $\|A\|=\min\{|m|,|M|\}$, see Theorem~3 in \cite{Loynes2}.

According to Theorem 4 in \cite{Loynes2},
the algebra $\cB^*(\cE)$ of bounded and adjointable operators on $\cE$ 
is a pre $C^*$-algebra and we have 
$\|A^*A\|=\|A\|^2$ for all $A\in\cB^*(\cE)$.

\subsection{VE-Modules over Ordered $*$-Algebras.}\label{ss:vemosa}
A VE-module $\cE$ over an ordered $*$-algebra $\cA$ is an ordered 
$*$-space over $\cA$, that is, (ve1)--(ve3) hold, subject to the following 
additional properties
\begin{itemize}
\item[(vem1)] $\cE$ is a right module over $\cA$, compatible with the 
multiplication with scalars: $\lambda (e a)=(\lambda e)a=e(\lambda a)$ for all
$\lambda\in\CC$, $e\in\cE$, and $a\in\cA$.
\item[(vem2)] $[e,fa+gb]_\cE=[e,f]_\cE a+[e,g]_\cE b$ for all $e,f,g\in\cE$ and 
all $a,b\in\cA$.
\end{itemize}

Given an ordered $*$-algebra $\cA$ and two VE-modules $\cE$ and $\cF$
over $\cA$, an operator $T\in\cL(\cE,\cF)$ is called a \emph{module map} if
\begin{equation*} T(ea)=T(e)a,\quad e\in\cE,\ a\in\cA.
\end{equation*}
Any operator $T\in\cL^*(\cE,\cF)$ is a module map: given arbitrary $e\in\cE$,
$f\in\cF$ and $a\in\cA$ we have
\begin{equation*}[T(ea),f]_\cF=[ea,T^*(f)]_\cE=a^*[e,T^*(f)]_\cE
= a^*[T(e),f]_\cF=[T(e)a,f]_\cF,
\end{equation*} hence $T$ is a module map. See \cite{Lance}, 
\cite{ManuilovTroitsky}, \cite{Skeide}, for the more special case of Hilbert 
modules over $C^*$-algebras.

\begin{examples}\label{ex:vemod}
Let $\cE$ and $\cF$ be 
two VE-spaces over the same ordered $*$-space $Z$. 

(1) The vector space $\cL^*(\cE,\cF)$ has a natural structure of VE-module 
over the ordered $*$-algebra $\cL^*(\cE)$, 
see Example~\ref{ex:les}, more precisely,
\begin{equation}\label{e:tese}
[T,S]=T^*S,\quad T,S\in\cL^*(\cE,\cF).
\end{equation}

(2) Let $\cS$ be a $*$-subspace of $\cL^*(\cE,\cF)$ and define a gramian
$[\cdot,\cdot]$ on $\cS$ by \eqref{e:tese}. Let $\cZ$ be the $*$-subspace of 
$\cL^*(\cE)$ generated by all operators $T^*S$, where $T,S\in\cS$. $\cZ$
has a natural structure of ordered $*$-space, where positivity of $T\in\cS$
is in the sense of \eqref{e:pos}. Thus, $\cS$ is a VE-space over $\cZ$ that,
in general, is not a VE-module. 
\end{examples}

\section{Linearisations for Invariant Kernels}\label{s:lik}

In this section we present the main dilation theorem for kernels. We start with
some preliminary results.

\subsection{Kernels with Values Adjointable Operators.}\label{ss:kvao}
Let $X$ be a nonempty set and let $\cH$ be a VE-space over the ordered 
$*$-space $Z$. 
A map $\fk\colon X\times X\ra \cL(\cH)$ is called a \emph{kernel} on $X$ and 
valued in 
$\cL(\cH)$. In case the kernel $\fk$ has all its values in $\cL^*(\cH)$, 
an \emph{adjoint} kernel 
$\fk^*\colon X\times X\ra\cL^*(\cH)$ can be associated by 
$\fk^*(x,y)=\fk(y,x)^*$ for all 
$x,y\in X$. The kernel $\fk$ is called \emph{Hermitian} if $\fk^*=\fk$.

Let $\cF=\cF(X;\cH)$ denote the complex vector space of all functions 
$f\colon X\ra \cH$ and let $\cF_0=\cF_0(X;\cH)$ be its subspace of those 
functions having 
finite support. A pairing $[\cdot,\cdot]_{\cF_0}\colon \cF_0\times \cF_0\ra Z$ 
can be defined by
\begin{equation}\label{e:prodge}
[g,h]_{\cF_0}=\sum_{y\in X} [g(y),h(y)]_\cH,\quad g,h\in\cF_0.
\end{equation} This pairing is clearly a $Z$-gramain on $\cF_0$,  hence 
$(\cF_0;[\cdot,\cdot]_{\cF_0})$ is a VE-space. 

Let us observe that the sum in \eqref{e:prodge} makes sense even 
when only one of the 
functions $g$ or $h$ has finite support, the other can be arbitrary in $\cF$. 
Thus, another pairing $[\cdot,\cdot]_\fk$ can be defined on $\cF_0$ by
\begin{equation}\label{e:prodka} [g,h]_\fk=
\sum_{x,y\in X}[\fk(y,x)g(x),h(y)]_\cH,
\quad g,h\in \cF_0.\end{equation} In general, the pairing $[\cdot,\cdot]_\fk$ 
is linear in the first variable and conjugate linear in the second variable. 
If, in addition, 
$\fk=\fk^*$ then the pairing $[\cdot,\cdot]_\fk$ is Hermitian as well, that is,
\begin{equation*} [g,h]_\fk=[h,g]_\fk^*,\quad g,h\in\cF_0.
\end{equation*} 

A \emph{convolution operator} $K\colon\cF_0\ra\cF$ can be associated to the 
kernel $\fk$ by
\begin{equation}\label{e:convop} (Kg)(y)=\sum_{x\in X} \fk(y,x)g(x),\quad g\in\cF_0,
\end{equation} and it is easy to see that $K$ is a linear operator.
There is a natural relation between the pairing $[\cdot,\cdot]_\fk$ 
and the convolution operator $K$ given by
\begin{equation*}[g,h]_\fk=[Kg,h]_{\cF_0},\quad g,h\in\cF_0.
\end{equation*} 
Given $n\in\NN$, the kernel $\fk$ is called \emph{$n$-positive} if for any 
$x_1,x_2,\ldots,x_n\in X$ and any $h_1,h_2,\ldots,h_n\in\cH$ we have
\begin{equation}\label{e:npos} \sum_{i,j=1}^n [\fk(x_i,x_j)h_j,h_i]_\cH\geq 0.
\end{equation}
 The kernel $\fk$ is called \emph{positive semidefinite} 
(or \emph{of positive type}) if it is $n$-positive for all natural numbers $n$.
The proof of the following lemma is the same as the proof of Lemma~3.1 from
\cite{Gheondea}.
 
The third assertion in the next result makes the connection with 
the kernels made up of bounded operators only as in \cite{Gheondea}.
 
 \begin{lemma}\label{l:twopos} Assume that the kernel $\fk\colon X\times 
X\ra\cL^*(\cH)$ is $2$-positive. Then:
 
 \nr{1} $\fk$ is Hermitian.
 
 \nr{2} If, for some $x\in X$, we have $\fk(x,x)=0$, then $\fk(x,y)=0$ for all 
$y\in X$.

\nr{3} Assume that, for $x,y\in X$ the operators $\fk(x,x)$ and $\fk(y,y)$ are
bounded. Then $\fk(x,y)$ and $\fk(y,x)=\fk(x,y)^*$ are bounded and
\begin{equation}\label{e:si}
\|\fk(x,y)\|^2\leq \|\fk(x,x)\|\, \|\fk(y,y)\|.
\end{equation}
In particular, if $\fk(x,x)\in\cB^*(\cE)$ for all $x\in X$, then
$\fk(y,x)\in\cB^*(\cE)$ for all $x,y\in X$.
\end{lemma}

\begin{proof} The proof of (1) and (2) is the same as the proof of 
Lemma~3.1 from \cite{Gheondea}.

(3) Assume that both operators $\fk(x,x)$ and $\fk(y,y)$ are
bounded, see Subsection~\ref{ss:ves}, hence
$\fk(x,x),\fk(y,y)\in\cB^*(\cE)$. 
If $\fk(y,y)=0$ then, by (2), $\fk(x,y)=0$ and 
$\fk(y,x)=\fk(x,y)^*=0$, hence bounded, and the inequality \eqref{e:si} holds
trivially.

Assume that $\fk(y,y)\neq 0$, hence $\|\fk(y,y)\|>0$. Since $\fk$ is
$2$-positive, for any $h,g\in\cH$ we have
\begin{equation}\label{e:twoposker}
 [\fk(x,x)h,h]+[\fk(x,y)g,h]+[\fk(y,x)h,g]+[\fk(y,y)g,g]\geq 0.
\end{equation}
We let $g=-\fk(x,y)^*h/\|\fk(y,y)\|$ in \eqref{e:twoposker}, take
into account \eqref{e:sai} and get
\begin{align*}\frac{2}{\|\fk(y,y)\|} [\fk(y,x)h,\fk(y,x)h] & \leq
  [\fk(x,x)h,h]+\frac{1}{\|\fk(y,y)\|^2} [\fk(y,y)\fk(y,x)h,\fk(y,x)h] \\
& \leq [\fk(x,x)h,h]+ \frac{\|\fk(y,y)\|}{\|\fk(y,y)\|^2}
  [\fk(y,x)h,\fk(y,x)h]\\
& =  [\fk(x,x)h,h]+ \frac{1}{\|\fk(y,y)\|}
  [\fk(y,x)h,\fk(y,x)h],
\end{align*}
hence
\begin{equation*}[\fk(y,x)h,\fk(y,x)h]\leq \|\fk(y,y)\| [\fk(x,x)h,h]\leq
  \|\fk(x,x)\| \, \|\fk(y,y)\| [h,h],
\end{equation*} which proves that $\fk(y,x)$ is a bounded operator and 
the inequality \eqref{e:si}.

\end{proof}

\begin{example}\label{e:kervesp}
This example is a generalisation of Example~\ref{ex:eos}.(6). 
Let $X$ be a nonempty set, $\cE$ 
be a VE-space over the ordered $*$-space $Z$. Let $\cK(X;\cE)$ 
be the vector space of all kernels 
$\fk:X\times X \ra \cL^*(\cE)$, 
and let $\cK(X;\cE)^+$ be the set of 
all positive semidefinite kernels. 
Then $\cK(X;\cE)^+$ is a cone and, 
by Lemma \ref{l:twopos}, 
it consists only of selfadjoint 
elements. If $\fk\in (\cK(X;\cE)^+\cap -\cK(X;\cE)^+)$, 
we obtain $[\fk(x,x)h,h]_{\cE}=0$ for 
all $x\in X$ and $h\in\cE$ by 
strictness of the cone of $Z$. 
Since $\fk(x,x)$ is a positive operator, hence selfadjoint, by means of the
analog of the polarisation formula \eqref{e:polar}, see the second part of 
Remark~\ref{r:polar}, it follows that $\fk(x,x)=0$ 
for any $x\in X$. Then, by Lemma~\ref{l:twopos} again, $\fk(x,y)=0$ 
for all $x,y\in X$, i.e. $\fk=0$. 
Therefore $\cK(X;\cE)$ is an ordered $*$-space 
with cone $\cK(X;\cE)^+$.  
A multiplication can be defined on $\cK(X;\cE)$: for $\fk,\fl\in\cK(X;\cE)$ we
let $(\fk\fl)(x,y)=\fk(x,y)\fl(x,y)$ for all $x,y\in X$. With respect to this 
multiplication, $\cK(X;\cE)$ is an ordered $*$-algebra.
\end{example}

Given an $\cL^*(\cH)$-valued kernel $\fk$ on a nonempty set $X$, for some 
VE-space $\cH$ on an ordered $*$-space $Z$,  a \emph{VE-space 
linearisation} or, equivalently, a
\emph{VE-space Kolmogorov decomposition} of $\fk$ is, by definition, 
a pair $(\cK;V)$, subject to the following conditions:
  
  \begin{itemize}
  \item[(kd1)] $\cK$ is a VE-space over the same ordered $*$-space $Z$.
  \item[(kd2)] $V\colon X\ra\cL^*(\cH,\cK)$ satisfies $\fk(x,y)=V(x)^*V(y)$ 
for all $x,y\in X$. \end{itemize}
The VE-space linearisation $(\cK;V)$ is called \emph{minimal} if
  \begin{itemize}
  \item[(kd3)] $\lin V(X)\cH=\cK$.
  \end{itemize}
Two VE-space linearisations $(V;\cK)$ and $(V';\cK')$ 
of the same kernel $\fk$ are 
called \emph{unitary equivalent} if there exists a unitary operator 
$U\colon \cK\ra\cK'$ such that $UV(x)=V'(x)$ for all $x\in X$.
  
The uniqueness of a minimal VE-space 
linearisation $(\cK;V)$ of a positive semidefinite kernel $\fk$, 
modulo unitary equivalence, 
follows in the usual way: if $(\cK';V')$ 
is another minimal VE-space linearisation of 
$\fk$, for arbitrary $x_1,\ldots,x_m,y_1,\ldots,y_n\in X$ and arbitrary 
$h_1,\ldots,h_m,g_1,\ldots,g_n\in \cH$, we have
\begin{align*} [\sum_{j=1}^m V(x_j)h_j, \sum_{i=1}^n V(y_i)g_i]_\cK 
& = \sum_{j=1}^m \sum_{i=1}^n [V(x_j)h_j,V(y_i)g_i]_\cK\\
& = \sum_{i=1}^n\sum_{j=1}^m
[\fk(y_i,x_j)h_j,g_i]_\cK \\
& =   \sum_{j=1}^m \sum_{i=1}^n [V'(x_j)h_j,V'(y_i)g_i]_{\cK'} \\
& = [\sum_{j=1}^m V'(x_j)h_j, \sum_{i=1}^n V'(y_i)g_i]_{\cE'},
\end{align*} hence $U\colon \lin V(X)\ra \lin V'(X)$ defined by 
\begin{equation}\label{e:udef}\sum_{j=1}^m 
V(x_j)h_j\mapsto \sum_{j=1}^m V'(x_j)h_j\end{equation} 
is a correctly everywhere defined linear operator, 
isometric and onto. Thus, $U$ is a 
VE-space isomorphism $U\colon \cK\ra\cK'$ and $UV(x)=V'(x)$ for all 
$x\in X$, by construction.

\subsection{Reproducing Kernel VE-Spaces.}\label{ss:rks}
Let $\cH$ be a VE-space over the ordered $*$-space $Z$, and let $X$ be a 
nonempty set. A VE-space $\cR$, 
over the same ordered $*$-space 
$Z$, is called an \emph{$\cH$-reproducing kernel VE-space on $X$} 
if there exists a Hermitian kernel $\fk\colon X\times X\ra\cL^*(\cH)$ 
such that the following axioms are satisfied:
\begin{itemize} 
\item[(rk1)] $\cR$ is a subspace of $\cF(X;\cH)$, with all algebraic operations.
\item[(rk2)] For all $x\in X$ and all $h\in\cH$, 
the $\cH$-valued function $\fk_x h=\fk(\cdot,x)h\in\cR$.
\item[(rk3)] For all $f\in\cR$ we have $[f(x),h]_\cH=[f,\fk_x h]_\cR$, for all 
$x\in X$ and $h\in \cH$.
\end{itemize}
As a consequence of (rk2), $\lin\{\fk_xh\mid x\in X,\ h\in\cH\}\subseteq \cR$.
The reproducing kernel VE-space $\cR$ is called \emph{minimal} if
the following property holds as well:
\begin{itemize}
\item[(rk4)] $\lin\{\fk_xh\mid x\in X,\ h\in \cH\}=\cR$.
\end{itemize}

Observe that if $\cR$ is an $\cH$-reproducing kernel VE-space on $X$, with 
kernel 
$\fk$, then $\fk$ is positive semidefinite and uniquely determined by $\cR$ 
hence we can talk about \emph{the} 
$\cH$-reproducing kernel $\fk$ corresponding to $\cR$.
On the other hand, a minimal reproducing kernel VE-space $\cR$ is 
uniquely determined by its reproducing kernel $\fk$.

The classical reproducing kernel Hilbert spaces, e.g.\ see \cite{Aronszajn}, 
are characterised, within the Hilbert function spaces, by the continuity of the 
evaluation functionals. In the following, we generalise this by showing that, 
in 
the absence of an analogue of the Riesz Representation Theorem, it is the 
adjointability which makes the difference. Letting $\cH$ be a VE-space over
an ordered $*$-space $Z$, for $X$ a nonempty set, an \emph{evaluation operator}
$E_x\colon \cF(X;\cH)\ra \cH$ can be defined for each $x\in X$ by letting 
$E_x f=f(x)$ for all $f\in\cF(X;\cH)$. Clearly, $E_x$ is linear.

\begin{proposition}\label{p:rks}
Let $X$ be a nonempty set, $\cH$ a VE-space over an ordered $*$-space 
$Z$, and let $\cR\subseteq \cF(X;\cH)$, with all algebraic operations, 
be a VE-space over $Z$. Then $\cR$ 
is an $\cH$-reproducing kernel VE-space if and only if, for all $x\in X$, the 
restriction of the evaluation operator $E_x$ to $\cR$ is adjointable as a
linear operator $\cR\ra \cH$.
\end{proposition}

\begin{proof} Assume first that $\cR$ is an $\cH$-reproducing kernel VE-space
on $X$ and let $\fk$ be its reproducing kernel.
For any $h\in\cH$ and any $f\in\cR$
\begin{equation}\label{e:aeo}
[E_xf,h]_\cH=[f(x),h]_\cH=[f,\fk_x h]_\cR.
\end{equation}
Since $\fk_x \in\cL(\cH,\cR)$, it follows that $E_x$ is adjointable and, in 
addition, $E_x^*=\fk_x$, for all $x\in X$.

Conversely, assume that, for all $x\in X$, the evaluation operator 
$E_x\in\cL^*(\cR,\cH)$. Equation \eqref{e:aeo} 
shows that, in order to show that $\cR$ is
a reproducing kernel VE-space, we should define the kernel $\fk$ 
in the following way:
\begin{equation}\label{e:drk}
\fk(y,x)=(E_x^*h)(y),\quad x,y\in X,\ h\in\cH.
\end{equation}
It is clear that $k(y,x)\colon \cH\ra\cH$ is a linear operator and observe
that $\fk_x h=E_x ^*h$ for all $x\in X$ and all $h\in \cH$. 
The reproducing property (rk3) holds:
\begin{equation*} [f(x),h]_\cH=[E_x f,h]_\cH=[f,E_x^* h]_\cR=[f,\fk_xh]_\cR,
\quad f\in\cR,\ h\in\cH,\ x\in X.
\end{equation*}
The axioms (rk1) and (rk2) are clearly satisfied, so it only remains to prove
that $\fk$ is a Hermitian kernel. To see this,
fix $x,y\in X$ and $h,l\in\cH$. Then
\begin{align*}[\fk(y,x)h,l]_\cH & = [(\fk_x h)(y),l]_\cH =[\fk_x h,\fk_y l]_\cR\\
& = [\fk_y l,\fk_x h]_\cR^*=[\fk(x,y)l,h]_\cR^*=[h,\fk(x,y)l]_\cR.
\end{align*}
Therefore, $\fk(y,x)$ is adjointable and $\fk(y,x)^*=\fk(x,y)$, hence $\fk$ is
a Hermitian kernel. We have proven that $\fk$ is the reproducing 
kernel of $\cR$.
\end{proof}

There is a very close connection between VE-space linearisations and
reproducing kernel VE-spaces.

\begin{proposition}\label{p:lvsrk}
Let $X$ be a nonempty set, $\cH$ a VE-space over an ordered $*$-space
$Z$, and let $\fk\colon X\times X\ra\cL^*(\cH)$ be a Hermitian kernel.

\nr{1} Any $\cH$-reproducing kernel VE-space $\cR$ with kernel $\fk$ is
a VE-space linearisation $(\cR;V)$ of $\fk$, with $V(x)=\fk_x$ for all $x\in X$.

\nr{2} For any minimal VE-space linearisation $(\cK;V)$ of $\fk$, letting 
\begin{equation}\label{e:redef} \cR=\{V(\cdot)^*f\mid f\in\cK\},
\end{equation} we obtain the minimal $\cH$-reproducing kernel VE-space with
reproducing kernel $\fk$.
\end{proposition}

\begin{proof}
(2)$\Ra$(1). Let $(\cK;\pi;V)$ be a minimal VE-space 
linearisation of the kernel $\fk$ on $X$.
Let $\cR$ be the set of all 
functions $X\ni x\mapsto V(x)^*f\in\cH$, in 
particular $\cR\subseteq \cF(X;\cH)$, and we endow $\cR$ with the algebraic 
operations inherited from the complex vector space $\cF(X;\cH)$. 

The correspondence
\begin{equation}\label{e:defuv} \cK\ni f\mapsto Uf=V(\cdot)^*f\in\cR
\end{equation} is bijective. By the definition of $\cR$, this correspondence 
is surjective. 
In order to verify that it is injective as well, let $f,g\in\cK$ be such that 
$V^*(\cdot)f=V^*(\cdot)g$. Then, for all $x\in X$ and all $h\in\cH$ we have
\begin{equation*}[V(x)^*f,h]_\cH=[V(x)^*g,h]_\cH,
\end{equation*} equivalently,
\begin{equation*}[f-g,V(x)h]_\cK=0,\quad x\in X,\ h\in\cH.
\end{equation*} By the minimality of the VE-space linearisation 
$(\cK;V)$ it follows that $g=f$. Thus, $U$ is a bijection.

Clearly, the bijective map $U$ defined at \eqref{e:defuv} is linear, 
hence a linear isomorphism of complex vector spaces $\cK\ra \cR$. 
On $\cR$ we introduce a $Z$-valued pairing
\begin{equation}\label{e:ufege}
[Uf,Ug]_\cR=[V(\cdot)^*f,V(\cdot)^*g]_\cR=[f,g]_\cK,\quad f,g\in\cK.
\end{equation} Then $(\cR;[\cdot,\cdot]_\cR)$ is a VE-space over $Z$ since, 
by \eqref{e:ufege}, we 
transported the $Z$-gramian from $\cK$ to $\cR$ or, in other words, we have 
defined on $\cR$ the $Z$-gramian that makes the linear isomorphism $U$ 
a unitary operator between the VE-spaces $\cK$ and $\cR$.

We show that
$(\cR;[\cdot,\cdot]_\cR)$ is an $\cH$-reproducing kernel VE-space with 
corresponding reproducing kernel $\fk$.
By definition, $\cR\subseteq \cF(X;\cH)$. On the other hand, since 
\begin{equation*}\fk_x(y)h=\fk(y,x)h=V(y)^*V(x)h,\mbox{ for all }x,y\in X\mbox{
    and all }h 
\in\cH,\end{equation*} taking into account that $V(x)h\in\cK$, 
by \eqref{e:redef} it follows that $\fk_x\in\cR$ for all $x\in X$. 
Further, for all $f\in\cR$, $x\in X$, and $h\in\cH$, we have
\begin{align*}[f,\fk_xh]_\cR & =[V(\cdot)^*g,\fk_x h]_\cR
=[V(\cdot)^*g,V(\cdot)^*V(x)h]_\cR
\\ & =[g,V(x)h]_\cK=[V(x)^*g,h]_\cH=[f(x),h]_\cH,
\end{align*} where $g\in\cK$ is the unique vector such that 
$V(x)^*g=f(x)$, which shows 
that $\cR$ satisfies the reproducing axiom as well.

(1)$\Ra$(2). Assume that $(\cR;[\cdot,\cdot]_\cR)$ is an $\cH$-reproducing 
kernel VE-space on $X$, with reproducing kernel $\fk$.
 We let $\cK=\cR$ and define
\begin{equation}\label{e:vexak}V(x)h=\fk_x h,\quad x\in X,\ h\in \cH.
\end{equation} Note that $V(x)\colon \cH\ra \cK$ is linear for all $x\in X$. 

We show that $V(x)\in\cL^*(\cH,\cK)$ for all $x\in X$. 
To see this, first note that, by the reproducing property,
\begin{equation}\label{e:fevex}
[f,V(x)h]_\cK=[f,\fk_xh]_\cR=[f(x),h]_\cH,\quad x\in X,\ h\in\cH.
\end{equation} Let us then, for fixed $x\in X$, consider the linear operator 
$W(x)\colon \cR=\cK\ra \cH$ defined by $W(x)f=f(x)$ for all $f\in\cR=\cK$. 
From 
\eqref{e:fevex} we conclude that $V(x)$ is adjointable and $V(x)^*=W(x)$ for 
all $x\in X$.

Finally, by the reproducing axiom, for all $x,y\in X$ and all $h,g\in\cH$ 
we have
\begin{equation*}[V(y)^*V(x)h,g]_\cH=[V(x)h,V(y)g]_\cR=[\fk_xh,\fk_yg]_\cR
=[\fk(y,x)h,g]_\cH,
\end{equation*} hence $V(y)^*V(x)=\fk(y,x)$ for all $x,y\in X$. 
Thus, $(\cK;V)$ is a VE-space linearisation of $\fk$ (actually, a minimal one). 
\end{proof}

\begin{remark}\label{r:lvsrk}
The proof of Proposition~\ref{p:lvsrk} provides an explicit correspondence 
between the class of minimal VE-space 
linearisations of $\fk$, identified by unitary equivalence, and the minimal 
$\cH$-reproducing kernel VE-space associated to $\fk$.
\end{remark}

\subsection{$*$-Representations on VE-Spaces
Associated to Invariant Kernels}\label{ss:r}
Let a (multiplicative) semigroup $\Gamma$ act on $X$, 
denoted by $\xi\cdot x$, for all $\xi\in\Gamma$ 
and all $x\in X$. By definition, we have 
\begin{equation}\label{e:action}
\alpha\cdot(\beta\cdot x)=(\alpha\beta)\cdot x\mbox{ for all }\alpha,
\beta\in \Gamma\mbox{ and all }x\in X.\end{equation} 
Equivalently, this means that we have a semigroup morphism 
$\Gamma\ni\xi\mapsto \xi\cdot \in G(X)$, where $G(X)$ denotes the 
semigroup, with respect to composition, of all maps $X\ra X$. 
In case the semigroups $\Gamma$ has a unit 
$\epsilon$, the action is called \emph{unital} if $\epsilon\cdot x=x$ 
for all $x\in X$, equivalently, $\epsilon\cdot=\mathrm{Id}_X$.

We assume further that $\Gamma$ is a $*$-semigroup, that is, there is an 
\emph{involution} $*$ on $\Gamma$; this means that $(\xi\eta)^*=\eta^* \xi^*$ 
and $(\xi^*)^*=\xi$ for all $\xi,\eta\in\Gamma$. 
Note that, in case $\Gamma$ has a unit $\epsilon$ then 
$\epsilon^*=\epsilon$.

Given a VE-space $\cH$ we are interested in those Hermitian kernels 
$\fk\colon X\times X\ra\cL^*(\cH)$ that are \emph{invariant} under the action 
of $\Gamma$ on $X$, that is,
\begin{equation}\label{e:invariant} \fk(y,\xi\cdot x)=\fk(\xi^*\cdot y,x)
\mbox{ for all }x,y\in X\mbox{ and all }\xi\in\Gamma.
\end{equation}
A triple $(\cK;\pi;V)$ is called an \emph{invariant VE-space linearisation} 
of  the kernel $\fk$ and the action of $\Gamma$ on $X$, shortly a
\emph{$\Gamma$-invariant VE-space linearisation} of $\fk$, if:
\begin{itemize}
\item[(ikd1)] $(\cK;V)$ is a VE-space linearisation of the kernel $\fk$.
\item[(ikd2)] $\pi\colon \Gamma\ra\cL^*(\cK)$ is a $*$-representation, that
  is, a multiplicative $*$-morphism.
\item[(ikd3)] $V$ and $\pi$ are related by the formula: 
$V(\xi\cdot x)= \pi(\xi)V(x)$, for all $x\in X$, $\xi\in\Gamma$.
\end{itemize}

\begin{remarks} (1)
Let $(\cK;\pi;V)$ be a $\Gamma$-invariant VE-space 
linearisation of the kernel 
$\fk$. Since $(\cK;V)$ is a VE-space linearisation
and taking into account the axiom (ikd3), we have
\begin{align}\label{e:keyxi}
\fk(y,\xi\cdot x)  & =V(y)^*V(\xi\cdot x)=V(y)^* \pi(\xi)V(x) \\
&=(\pi(\xi^*)V(y))^* 
V(x)= \fk(\xi^*\cdot y,x),\quad x,y\in X,\ \xi\in\Gamma,\nonumber
\end{align} hence $\fk$ is invariant under the action of $\Gamma$ on $X$.

(2) Observe that, if the action of $\Gamma$ on $X$ is unital then, for a
$\Gamma$-invariant VE-space linearisation $(\cK;\pi;V)$, the two conditions
$\fk(x,y)=V(x)^*V(y)$, $x,y\in X$, and $V(\xi\cdot x)=\pi(\xi)V(x)$,
$\xi\in\Gamma$ and $x\in X$, can be equivalently combined into two slightly
different conditions, namely, $\pi$ unital and
$\fk(x,\xi\cdot y)=V(x)^*\pi(\xi)V(y)$, $\xi\in\Gamma$ and $x,y\in X$.
\end{remarks}

If, in addition to the axioms (ikd1)--(ikd3), the triple $(\cK;\pi;V)$ 
has the property
\begin{itemize}
\item[(ikd4)] $\lin V(X)\cH=\cK$,
\end{itemize} that is, the VE-space linearisation $(\cK;V)$ is minimal,
then $(\cK;\pi;V)$ is called a \emph{minimal
$\Gamma$-invariant VE-space linearisation} of $\fk$ and the action of 
$\Gamma$ on $X$.

Minimal invariant VE-space linearisations have a built-in 
linearity property; the proof is the same with that of 
Proposition~4.1 in \cite{Gheondea}.

\begin{proposition}\label{P:vhinvkolmolinear} Assume that, given 
a VE-space adjointable operator valued kernel $\fk$, 
invariant under the action of the $*$-semigroup 
$\Gamma$ on $X$, for some fixed $\alpha,\beta,\gamma\in\Gamma$ 
we have 
$\fk(y,\alpha\cdot x)+\fk(y,\beta\cdot x)=\fk(y,\gamma\cdot x)$ 
for all $x,y\in X$. Then, for 
any minimal invariant VE-space linearisation $(\cK;\pi;V)$ of $\fk$, the 
representation satisfies $\pi(\alpha)+\pi(\beta)=\pi(\gamma)$.
\end{proposition}

Two $\Gamma$-invariant VE-space linearisations $(\cK;\pi;V)$ and 
$(\cK';\pi';V')$, 
of the same Hermitian kernel $\fk$, are called \emph{unitary equivalent} 
if there exists a unitary operator $U\colon \cK\ra\cK'$ such that 
$U\pi(\xi)=\pi'(\xi)U$ for all $\xi\in\Gamma$, 
and $UV(x)=V'(x)$ for all $x\in X$. Let us note that, in case both of these 
invariant VE-space linearisations are minimal, then this is equivalent 
with the requirement that the VE-space linearisations $(\cK;V)$ and 
$(\cK';V')$ are unitary equivalent.

The main result of this article
is the following theorem. It is stated in terms of both linearisations and 
reproducing kernels and  the proof points out essentially a 
reproducing kernel and operator range construction.

\begin{theorem}\label{t:vhinvkolmo} Let $\Gamma$ be a $*$-semigroup 
that acts 
on the nonempty set $X$ and let $\fk\colon X\times X\ra\cL^*(\cH)$ be a kernel, 
for some VE-space $\cH$ over an ordered $*$-space $Z$. 
The following assertions are equivalent:

\begin{itemize}
\item[(1)] $\fk$ is positive semidefinite, in the sense of \eqref{e:npos}, 
and invariant under the action of $\Gamma$ on $X$, that is, 
\eqref{e:invariant} holds.
\item[(2)] $\fk$ has a $\Gamma$-invariant VE-space linearisation 
$(\cK;\pi;V)$.
\item[(3)] $\fk$ admits an $\cH$-reproducing kernel VE-space $\cR$ and 
there exists a $*$-representation $\rho\colon \Gamma\ra\cL^*(\cR)$ such that 
$\rho(\xi)\fk_xh=\fk_{\xi\cdot x}h$ for all $\xi\in\Gamma$, $x\in X$, $h\in\cH$.
\end{itemize}

In addition, in case any of the assertions \emph{(1)}, \emph{(2)}, or
\emph{(3)} holds,  
then a minimal $\Gamma$-invariant VE-space linearisation can be constructed,
any minimal $\Gamma$-invariant VE-space linearisation is unique up to unitary 
equivalence, 
a pair $(\cR;\rho)$ as in assertion \emph{(3)} with $\cR$ 
minimal can be always obtained and, in this case, it is uniquely 
determined by $\fk$ as well.
\end{theorem}

\begin{proof} (1)$\Ra$(2).
Assuming that $\fk$ is positive semidefinite, by 
Lemma~\ref{l:twopos}.(1) it follows that $\fk$ is Hermitian, that is, 
$\fk(x,y)^*=\fk(y,x)$ for all $x,y\in X$.
We consider the convolution 
operator $K$ defined at \eqref{e:convop} and let $\cG=\cG(X;\cH)$ be its 
range, more precisely,
\begin{align}\label{e:fezero} \cG & =\{f\in\cF\mid f=Kg\mbox{ for some }
g\in\cF_0\}\\
& = \{f\in\cF\mid f(y)=\sum_{x\in X} \fk(y,x)g(x)\mbox{ for some }g\in\cF_0
\mbox{ and all } 
x\in X\}.\nonumber
\end{align}
A pairing $[\cdot,\cdot]_{\cG}\colon \cG\times \cG\ra Z$ can be defined by
\begin{align}\label{e:ipfezero} 
[e,f]_{\cG} &=[Kg,h]_{\cF_0}=\sum_{y\in X}[e(y),h(y)]_{\cH}\\
& =
\sum_{x,y\in X} [\fk(y,x)g(x),h(y)]_\cH,\nonumber
\end{align} where $f=Kh$ and $e=Kg$ for some $g,h\in\cF_0$. 
We observe that
\begin{align*} [e,f]_{\cG} & =\sum_{y\in X}[e(y),h(y)]_{\cH}
=\sum_{x,y\in X} [\fk(y,x)g(x),h(y)]_\cH\\ & = \sum_{x,y\in X} [g(x),\fk(x,y)h(y)]_\cH
=\sum_{x\in X} [g(x),f(x)]_\cH,
\end{align*} which shows that the definition in \eqref{e:ipfezero} 
is correct, that is, independent of $g$ and $h$ such that 
$e=Kg$ and $f=Kh$.

We claim that $[\cdot,\cdot]_{\cG}$ is a $Z$-valued gramian, that is, 
it satisfies all the requirements (ve1)--(ve3). 
The only fact that needs a proof is $[f,f]_{\cG}=0$ 
implies $f=0$ and this follows by Lemma~\ref{l:sesqui}. 

Thus, $(\cG;[\cdot,\cdot]_{\cG})$ is a VE-space that we denote by $\cK$.
For each $x\in X$ we define $V(x)\colon \cH\ra\cG$ by
\begin{equation}\label{e:defvex} V(x) h=Kh_x,\quad h\in\cH,
\end{equation} where $h_x=\delta_x h\in\cF_0$ is the function that takes the
value $h$ at $x$ and is null elsewhere.
Equivalently,
\begin{equation}\label{e:vexa} (V(x) h)(y)=(Kh_x)(y)
=\sum_{z\in X} \fk(y,z)(h_x)(z)=\fk(y,x)h,\ \ y\in X.
\end{equation}

Note that $V(x)$ is an operator from the VE-space $\cH$ to the 
VE-space $\cG=\cK$ and it remains to show that $V(x)$ is adjointable 
for all $x\in X$. To see this, let us fix $x\in X$ 
and take $h\in\cH$ and $f\in\cG$ arbitrary. Then,
\begin{equation}\label{e:vexadj} [V(x)h,f]_{\cG}=\sum_{y\in X} 
[(h_x)(y),f(y)]_\cH=[h,f(x)]_\cH,
\end{equation} which shows that $V(x)$ is adjointable and that its adjoint
$V(x)^*$ is the operator $\cG\ni f\mapsto f(x)\in \cH$ of evaluation at $x$.

On the other hand, for any $x,y\in X$, by
\eqref{e:vexa}, we have
\begin{equation*} V(y)^*V(x)h=(V(x)h)(y) =\fk(y,x)h,\quad h\in\cH,
\end{equation*} hence $(V;\cK)$ is a VE-space linearisation of $\fk$. We prove that 
it is minimal as well. To see this, note that a typical element in the linear
span of $V(X)\cH$ is, for arbitrary $n\in\NN$, $x_1,\ldots,x_n\in X$, and
$h_i,\ldots,h_n\in \cH$,
\begin{align*} \sum_{j=1}^n V(x_j)h_j & = \sum_{j=1}^n Kh_{j,x_j} \\
& = \sum_{j=1}^n \sum_{y\in X} \fk(\cdot,y)h_{j,x_j}(y) =\sum_{j=1}^n
  \fk(\cdot,x_j)h_j,
\end{align*} and then take into account that $\cG$ is the range of the
convolution operator $K$ defined at \eqref{e:convop}.
The uniqueness of the minimal VE-space linearisation $(V;\cK)$ just 
constructed follows as in \eqref{e:udef}.

For each $\xi\in\Gamma$ we let 
$\pi(\xi)\colon \cF\ra\cF$ be 
defined by
\begin{equation} (\pi(\xi)f)(y)=f(\xi^*\cdot y),\quad y\in X,\ \xi\in\Gamma.
\end{equation}
We prove that $\pi(\xi)$ leaves $\cG$ invariant. 
To see this, let $f\in\cG$, that is, $f=Kg$ for 
some $g\in\cF_0$ or, even more explicitly, by \eqref{e:fezero},
\begin{equation}f(y)=\sum_{x\in X} \fk(y,x)g(x),\quad y\in X.
\end{equation}
Then, 
\begin{align}\label{e:fexistar} 
f(\xi^*\cdot y) & =\sum_{x\in X}\fk(\xi^*\cdot y,x)g(x)\\
& =\sum_{x\in X}\fk(y,\xi
\cdot x)g(x)=\sum_{z\in X} \fk(y,z)g^\xi(z),\nonumber
\end{align} where,
\begin{equation} g^\xi(z)=\begin{cases} 0,& \mbox{ if }\xi\cdot x=z
\mbox{ has no solution }x\in \supp g,\\ \sum\limits_{\xi\cdot x=z} g(x), & 
\mbox{ otherwise.}
\end{cases}
\end{equation} Since $g^\xi \in \cF_0$,
it follows
that $\pi(\xi)$ leaves $\cG$ invariant. In the following we denote by 
the same symbol $\pi(\xi)$ the map $\pi(\xi)\colon \cG\ra\cG$.

We prove that $\pi$ is a representation of the semigroup $\Gamma$ on 
the complex vector space $\cG$, that is, 
\begin{equation}\label{e:piab}
\pi(\alpha\beta)f=\pi(\alpha)\pi(\beta)f,\quad \alpha,\beta\in \Gamma,\ 
f\in\cG.
\end{equation}
To see this, let $f\in\cG$ be fixed and denote $h=\pi(\beta)f$, that is, 
$h(y)=f(\beta^*\cdot y)$ for all $y\in X$. 
Then $\pi(\alpha)\pi(\beta)f=\pi(\alpha)h$, that 
is, $(\pi(\alpha)h)(y)=h(\alpha^*\cdot y)=h(\beta^*\alpha^*\cdot y)
=h((\alpha\beta)^*\cdot 
y)=(\pi(\alpha\beta))(y)$, for all $y\in X$, which proves \eqref{e:piab}

We show that $\pi$ is actually a $*$-representation, that is,
\begin{equation}\label{e:piastar} [\pi(\xi)f,f']_{\cG}
=[f,\pi(\xi^*)f']_{\cG},\quad f,f'\in \cG.
\end{equation}
To see this, let $f=Kg$ and $f'=Kg'$ for some $g,g'\in\cF_0$. Then, recalling 
\eqref{e:ipfezero} and \eqref{e:fexistar}, 
\begin{align*} [\pi(\xi)f,f']_{\cG} & = \sum_{y\in X}[f(\xi^*y),g'(y)]_\cH 
= \sum_{x,y\in X}[\fk(\xi^*\cdot y,x)g(x),g'(y)]_\cH \\
& = \sum_{x,y\in X}[\fk(y,\xi\cdot x)g(x),g'(y)]_\cH 
= \sum_{x,y\in X}[g(x),\fk(\xi\cdot x,y)g'(y)]_\cH \\
& = \sum_{x\in X}[g(x),f'(\xi\cdot x)]_\cH
 = [f,\pi(\xi^*)f']_\cH,
\end{align*} and hence the formula \eqref{e:piastar} is proven.

In order to show that the axiom (ikd3) holds as well, we use \eqref{e:vexa}. 
Thus, for all $\xi\in\Gamma$, $x,y\in X$, $h\in\cH$, and taking into account 
that $\fk$ is 
invariant under the action of $\Gamma$ on $X$, we have
\begin{align}\label{e:vexic} 
(V(\xi\cdot x)h)(y) & =\fk(y,\xi\cdot x)h=\fk(\xi^*\cdot y,x)h \\
 & =(V(x)h)(\xi^*\cdot y)=(\pi(\xi) V(x)h)(y),\nonumber
\end{align} which proves (ikd3). Thus, $(\cK;\pi;V)$, here constructed, 
is a $\Gamma$-invariant VE-space linearisation of the Hermitian kernel $\fk$.
Note that $(\cK;\pi;V)$ is minimal, that is, the axiom (ikd4) holds, 
since the VE-space linearisation
$(\cK;V)$ is minimal.

Let $(\cK';\pi';V')$ be another minimal invariant VE-space linearisation of 
$K$. We consider the unitary operator $U\colon \cK\ra\cK'$ defined as in 
\eqref{e:udef} and we already know that $UV(x)=V'(x)$ for all $x\in X$. Since, 
for any $\xi\in\Gamma$, $x\in X$, and $h\in\cH$, we have
\begin{equation*}U\pi(\xi)V(x)h=UV(\xi\cdot x)h=V'(\xi\cdot x)h=\pi'(\xi)V'(x)h
=\pi'(\xi)UV(x)h,
\end{equation*} and taking into account the minimality, it follows that 
$U\pi(\xi)=\pi'(\xi)U$ for all $\xi\in\Gamma$.

(2)$\Ra$(1). Let $(\cK;\pi;V)$ be a $\Gamma$-invariant VE-space linearisation of $\fk$.
Then
\begin{align*} \sum_{j,i=1}^n [\fk(x_i,x_j)h_j,h_i]_\cH & = \sum_{j,i=1}^n [V(x_i)^* 
V(x_j)h_j,h_i]_\cH \\ & = [\sum_{j=1}^n V(x_j)h_j,\sum_{j=1}^n V(x_j)h_j]_\cH\geq 0,
\end{align*} for all $n\in\NN$, $x_1,\ldots,x_n\in X$, and $h_1,\ldots,h_n\in
\cH$, hence $\fk$ is positive semidefinite.
It was shown in \eqref{e:keyxi} that $\fk$ is invariant under 
the action of $\Gamma$ on $X$.

(2)$\Ra$(3). This follows from Proposition~\ref{p:lvsrk} with the following
observation: with notation as in the proof of that proposition, 
for all $x,y\in X$ and $h\in\cH$ we have
\begin{equation*} \fk_{\xi\cdot x}(y)h=\fk(y,\xi\cdot x)h=V(y)^*V(\xi\cdot x)h
=V(y)^*\pi(\xi) 
V(x)h,
\end{equation*} hence, letting $\rho(\xi)=U\pi(\xi)U^{-1}$, 
where $U\colon\cK\ra\cR$ is 
the unitary operator defined as in \eqref{e:defuv}, we obtain a 
$*$-representation of $\Gamma$ on the VH-space $\cR$ 
such that $\fk_{\xi\cdot x}=\rho(\xi)\fk_x$ for all 
$\xi\in\Gamma$ and $x\in X$.

(3)$\Ra$(2). Let $\rho\colon\Gamma\ra\cL^*(\cR)$ is a $*$-representation 
such that 
$\fk_{\xi\cdot x}=\rho(\xi)\fk_x$ for all $\xi\in\Gamma$ and $x\in X$.
Again, we use Proposition~\ref{p:lvsrk}. Letting $\pi=\rho$, it is 
then easy to see that $(\cR;\pi;V)$ is a $\Gamma$-invariant VE-space 
linearisation of the kernel $\fk$.
\end{proof}

\begin{remarks}\label{r:vhinvkolmo}
 (1) Given $\fk\colon X\times X\ra\cL^*(\cH)$ a positive 
semidefinite kernel, 
as a consequence of Theorem~\ref{t:vhinvkolmo} 
we can denote, without any ambiguity, by $\cR_\fk$ 
the unique minimal 
$\cH$-reproducing kernel VE-space on $X$ associated to $\fk$.

(2) The 
construction in the proof of (1)$\Ra$(2) in Theorem~\ref{t:vhinvkolmo} 
is essentially a minimal
$\cH$-reproducing kernel VE-space 
one. 
More precisely,  we first note that,
for arbitrary $f\in\cF(X;\cH)$, $f=Kg$ with $g\in\cF_0(X;\cH)$, we have
\begin{equation}f=\sum_{x\in X} \fk(y,x)g(x)=\sum_{x\in X} \fk_x(y) g(x),
\end{equation} hence $\cG(X;\cH)=\lin\{\fk_x h\mid x\in X,\ h\in\cH\}$. 
Then, for 
arbitrary $f\in\cG$ we have
\begin{align*}[f,\fk_x h]_\cK & =[f,\fk_x h]_{\cG}=[f,K h_x]_{\cG} 
= \sum_{y\in X}
[f(y), (h_x)(y)]_\cH\\ 
&= [f(x),h]_\cH=[f,\fk_x h]_{\cR(K)},\quad x\in X,\ h\in\cH,
\end{align*} hence 
$[\cdot,\cdot]_{\cK}=[\cdot,\cdot]_{\cR(K)}$ on $\cG(X;\cH)=
\lin\{\fk_x h\mid x\in X,\ h\in\cH\}$, that coincides with both $\cK$ and
$\cR(K)$. 
Therefore, we can take $\cK=\cR(K)=\cG(X;\cH)$ to be a VE-space, with 
the advantage that it consists entirely of $\cH$-valued functions on $X$.

This idea was used in \cite{Gheondea} as well and the source 
of inspiration is \cite{BSzNagy}.
\end{remarks}

\section{Results that Theorem~\ref{t:vhinvkolmo} Unifies}\label{s:rtu}

In this section we obtain, as consequences of the main result, different 
versions of known dilation theorems in non topological versions.

\subsection{Positive Semidefinite Maps on $*$-Semigroups.}
Given a VE-space $\cH$ over an ordered $*$-space $Z$ 
and a $*$-semigroup $\Gamma$, a map 
$\phi\colon \Gamma\ra\cL^*(\cH)$ is called \emph{positive semidefinite} 
or \emph{of positive type} if, for all $n\in\NN$, $\xi_1,\ldots,\xi_n\in\Gamma$ 
and $h_1,\ldots,h_2\in\cH$, we have
\begin{equation}\label{e:ispos} \sum_{i,j=1}^n [\phi(\xi_i^*\xi_j)h_j,h_i]_\cH
\geq 0.
\end{equation}

Given a map $\phi\colon\Gamma\ra\cL^*(\cH)$ we consider the kernel 
$\fk\colon\Gamma\times \Gamma\ra\cL^*(\cH)$ defined by
\begin{equation}\label{e:kaunu} \fk(\alpha,\beta)=\phi(\alpha^*\beta),
\quad \alpha,\beta\in\Gamma,
\end{equation} and observe that $\phi$ is positive semidefinite, 
in the sense of 
\eqref{e:ispos}, if and only if 
$\fk$ is positive semidefinite, in the sense 
of \eqref{e:npos}.

On the other hand, considering the action of $\Gamma$ on itself by left 
multiplication, the kernel $\fk$, as defined at \eqref{e:kaunu}, is 
$\Gamma$-invariant, in the sense of \eqref{e:invariant}. Indeed,
\begin{equation*}\fk(\xi,\alpha\cdot\zeta)=\phi(\xi^*\alpha\zeta)
=\phi((\alpha^*\xi)^*\zeta)=\fk(\alpha^*\cdot\xi,\zeta),\quad \alpha,\xi,\zeta
\in\Gamma.
\end{equation*}
Therefore, the following corollary is a direct consequence of 
Theorem~\ref{t:vhinvkolmo}.

\begin{corollary}\label{c:sznagy1} Let $\phi\colon \Gamma\ra \cL^*(\cH)$ be a 
map, for some $*$-semigroup $\Gamma$ and some VE-space $\cH$ over an
ordered $*$-space $Z$. The following assertions are equivalent:
\begin{itemize}
\item[(1)] The map $\phi$ is positive semidefinite.
\item[(2)] There exists a VE-space $\cK$ over $Z$, 
a map $V\colon\Gamma\ra\cL^*(\cH,\cK)$, and a $*$-representation 
$\pi\colon \Gamma\ra\cL^*(\cK)$, such that:
\begin{itemize}
\item[(i)] $\phi(\xi^*\zeta)=V(\xi)^*V(\zeta)$ for all $\xi,\zeta\in\Gamma$.
\item[(ii)] $V(\xi\zeta)=\pi(\xi)V(\zeta)$ for all $\xi,\zeta\in\Gamma$.
\end{itemize}
\end{itemize}
In addition, if this happens, then the triple $(\cK;\pi;V)$ can always 
be chosen 
minimal, in the sense that $\cK$ is the linear span of the set 
$V(\Gamma)\cH$, and any two minimal triples as before are unique, modulo 
unitary equivalence.
\begin{itemize}
\item[(3)] There exist an $\cH$-reproducing kernel VE-space $\cR$ on 
$\Gamma$ and a 
$*$-representation $\rho\colon\Gamma\ra\cL^*(\cR)$ such that:
\begin{itemize}
\item[(i)] $\cR$ has the reproducing kernel $\Gamma\times\Gamma
\ni(\xi,\zeta)\mapsto \phi(\xi^*\zeta)\in\cL^*(\cH)$.
\item[(ii)] $\rho(\alpha)\phi(\cdot\xi)h=\phi(\cdot\alpha\xi)h$ for all $
\alpha,\xi\in\Gamma$ and $h\in\cH$.
\end{itemize}
\end{itemize}
In addition, the reproducing kernel VE-space $\cR$ as in (3) can be 
always constructed minimal and in this case it is uniquely determined by 
$\phi$.
\end{corollary}

As can be observed from condition (2).(i) in Corollary~\ref{c:sznagy1}, we
do not have a representation of $\phi$ on the whole $*$-semigroup 
$\Gamma$ but only on its $*$-subsemigroup 
$\{\xi^*\zeta\mid \xi,\zeta\in\Gamma\}$, which may be strictly smaller than 
$\Gamma$. This situation can be remedied, for example,
in case the $*$-semigroup $\Gamma$ has a unit, when the previous corollary 
takes a form similar with B.~Sz.-Nagy Theorem, cf.\ \cite{BSzNagy}.

\begin{corollary}\label{c:sznagy2} Assume that the $*$-semigroup 
$\Gamma$ has a unit $\epsilon$.  
Let $\phi\colon \Gamma\ra \cL^*(\cH)$ be a 
map, for some VE-space $\cH$ over an ordered $*$-space $Z$. 
The following assertions are equivalent:
\begin{itemize}
\item[(1)] The map $\phi$ is positive semidefinite.
\item[(2)] There exist a VE-space $\cK$ over $Z$, 
a linear operator $W\in\cL^*(\cH,\cK)$, and a unital $*$-representation 
$\pi\colon \Gamma\ra\cL^*(\cK)$, such that:
\begin{equation}\label{e:kapsa} \phi(\alpha)=W^*\pi(\alpha)W,\quad 
\alpha\in\Gamma.
\end{equation}
\end{itemize}
In addition, if this happens, then the triple $(\cK;\pi;V)$ can always be chosen 
minimal, in the sense that $\cK$ is the linear span of the set 
$\pi(\Gamma)W\cH$, and any two minimal triples as before are unique, 
modulo unitary equivalence.
\end{corollary}

\subsection{Positive Semidefinite Linear Maps.}\label{ss:pslm}
Given a $*$-algebra $\cA$, a linear map $\phi\colon\cA\ra\cL^*(\cH)$, for 
some VE-space $\cH$ over an ordered $*$-space $Z$, 
is called \emph{positive semidefinite} if for all $a_1,\ldots,a_n\in\cA$ and all 
$h_1,\ldots,h_n\in \cH$ we have
\begin{equation}\label{e:lpsd}
\sum_{i,j=1}^n [\phi(a_i^*a_j)h_j,h_i]_\cH\geq 0,
\end{equation} where the inequality is understood in $Z$ with respect to the 
given cone $Z^+$ and the underlying partial order, see 
Subsection~\ref{ss:oss}. Observe that for a Hermitian linear map 
$\phi\colon\cA\ra\cL^*(\cH)$ one can define a Hermitian kernel  
$\fk_\phi\colon \cA\times \cA\ra \cL^*(\cH)$  by letting
\begin{equation*} \fk_\phi(a,b)=\phi(a^*b),\quad a,b\in\cA.
\end{equation*} Also, observe that the $*$-algebra $\cA$ can be viewed as 
a multiplicative $*$-semigroup and, letting $\cA$ act on itself by multiplication, 
the kernel $\fk_\phi$ is invariant under this action. With this notation, another 
consequence of Theorem~\ref{t:vhinvkolmo} is the following

\begin{corollary}\label{c:stinespring1} Let $\phi\colon \cA\ra \cL^*(\cH)$ be a 
linear 
map, for some $*$-algebra $\cA$ and some VE-space $\cH$ over an 
ordered $*$-space $Z$. The following assertions are equivalent:
\begin{itemize}
\item[(1)] The map $\phi$ is positive semidefinite.
\item[(2)] There exist a VE-space $\cK$ over the ordered $*$-space $Z$, 
a linear map $V\colon\cA\ra\cL^*(\cH,\cK)$, and a $*$-representation 
$\pi\colon \cA\ra\cL^*(\cK)$, such that:
\begin{itemize}
\item[(i)] $\phi(a^*b)=V(a)^*V(b)$ for all $a,b\in\cA$.
\item[(ii)] $V(ab)=\pi(a)V(b)$ for all $a,b\in\cA$.
\end{itemize}
\end{itemize}
In addition, if this happens, then the triple $(\cK;\pi;V)$ can always be chosen 
minimal, in the sense that $\cK$ is the linear span of the set 
$V(\cA)\cH$, and any two minimal triples as before are unique, modulo 
unitary equivalence.
\begin{itemize}
\item[(3)] There exist an $\cH$-reproducing kernel VE-space $\cR$ on 
$\cA$ and a 
$*$-representation $\rho\colon\cA\ra\cL^*(\cR)$ such that:
\begin{itemize}
\item[(i)] $\cR$ has the reproducing kernel $\cA\times\cA
\ni(a,b)\mapsto \phi(a^*b)\in\cL^*(\cH)$.
\item[(ii)] $\rho(a)\phi(\cdot b)h=\phi(\cdot ab)h$ for all $
a,b\in\cA$ and $h\in\cH$.
\end{itemize}
\end{itemize}
In addition, the reproducing kernel VE-space $\cR$ as in (3) can be 
always constructed minimal and in this case it is uniquely determined by 
$\phi$.
\end{corollary}

In case the $*$-algebra has a unit, the previous corollary yields a Stinespring 
type Representation Theorem, cf.\ \cite{Stinespring}, or its generalisations
\cite{Gheondea}. 
More precisely, letting $e$ denote the unit of the $*$-
algebra $\cA$ and with the notation as in Corollary~\ref{c:stinespring1}.(2), 
letting $W=V(e)$, we have

\begin{corollary}\label{c:stinespring2} 
Let $\cA$ be a unital $*$-algebra $\cA$ and 
$\phi\colon\cA\ra\cL^*(\cH)$ a linear map, for 
some VE-space $\cH$ over an ordered $*$-space $Z$. The following 
assertions are equivalent:
\begin{itemize}
\item[(i)] $\phi$ is positive semidefinite.
\item[(ii)] There exist $\cK$ a VE-space over 
the same ordered $*$-space $Z$, a $*$-representation
 $\pi\colon\cA\ra\cL^*(\cK)$, and $W\in\cL^*(\cH,\cK)$ such that 
 \begin{equation}\phi(a)=W^*\pi(a)W\quad a\in\cA.\end{equation}
\end{itemize}
In addition, if this happens, then the triple $(\cK;\pi;W)$ can always be chosen 
minimal, in the sense that $\cK$ is the linear span of the set 
$\pi(\cA)W\cH$, and any two minimal triples as before are unique, 
modulo unitary equivalence.
\end{corollary}

\begin{remarks}\label{r:cp}
(1) In dilation theory, one encounters also the notion of 
\emph{completely positive}, e.g.\ see \cite{Paulsen}. In our setting, we
can consider a linear map $\phi\colon\cV\ra\cL^*(\cE)$, 
where $\cV$ is a $*$-space and $\cE$ is some VE-space over
an ordered $*$-space $Z$. For each $n$ one can consider the 
$*$-space $M_n(\cV)$ of all $n\times n$ matrices with entries in $\cV$. 
Then the \emph{$n$-th amplification map} 
$\phi_n\colon M_n(\cV)\ra M_n(\cL^*(\cE))=\cL^*(\cE^n)$ is defined by
\begin{equation}\label{e:phn}
\phi_n([a_{i,j}]_{i,j=1}^n)=[\phi(a_{i,j})]_{i,j=1}^n,\quad [a_{i,j}]_{i,j=1}^n\in M_n(\cV).
\end{equation}
Basically, $\phi$ would be called \emph{completely positive} if $\phi_n$
is "positive" for all $n$, where "positive" should mean that, whenever 
$[a_{i,j}]_{i,j=1}^n$ is "positive" in $M_n(\cV)$ then $\phi_n([a_{i,j}]_{i,j=1}^n)$
is positive in $M_n(\cL^*(\cE))$. Since positivity in $M_n(\cL^*(\cE))$ 
is perfectly defined, see Example~\ref{ex:les}.(4), 
the only problem is to define positivity 
in $M_n(\cV)$. One of the possible approaches, e.g.\ see 
\cite{EsslamzadehTurowska}, is
to assume $\cV$ be a \emph{matrix quasi ordered $*$-space},  that is, there
exists $\{C_n\}_{n\geq 1}$ a \emph{matrix quasi ordering} of $\cV$, 
in the following sense
\begin{itemize}
\item[(mo1)] For each $n\in\NN$, $C_n$ is a cone on $M_n(\cV)$.
\item[(mo2)] For each $m,n\in\NN$ and each $m\times n$ matrix with 
complex entries, we have $T^*C_mT\subseteq C_n$, where multiplication is
the usual matrix multiplication.
\end{itemize}
In the special case when (mo1) is changed such 
that for each $n\in\NN$, the cone
$C_n$ is strict, one has the concept of \emph{matrix ordering} and, 
respectively, of \emph{matrix ordered $*$-space}, e.g.\ see
\cite{PaulsenTodorovTomforde}. For example, $\cL^*(\cE)$ has a natural 
structure of matrix ordered $*$-algebra, see Example~\ref{ex:les}.(4).
Observe that, in the latter case, 
each $M_n(\cV)$ is an ordered $*$-space hence, in 
this case, the concept of completely positive map 
$\phi\colon\cV\ra \cL^*(\cE)$ makes perfectly sense. In the former case, that 
of matrix quasi ordered $*$-space $\cV$,
the concept of completely positive map $\phi$ makes sense as well.

(2) Assuming that instead of $\cV$ we have a $*$-algebra $\cA$ and that
the concept of a completely positive map $\phi\colon \cA\ra\cL^*(\cE)$ is 
defined, a natural question is what is the relation of this concept with that of 
positive semidefinite map $\phi$. By inspection, it can be observed that, in 
order to relate the two concepts, the matrix (quasi) ordering on $\cA$ should  
be related with that of $*$-positivity, see Remark~\ref{r:isa}. More precisely,
observe first that $*$-positivity provides in a natural way 
a matrix quasi ordering of $\cA$. Then, one can prove that if $\phi$ is 
completely positive, with definition as in item (1) and with respect to the $*$-
positivity, then $\phi$ is positive semidefinite, with definition as in 
\eqref{e:lpsd}. The converse is even more problematic, depending on
whether any $*$-positive matrix $[a_{i,j}]_{i,j=1}^n$ can be represented 
as a sum of matrices $a^*a$, where $a$ is a special matrix
with only one non-null row. This special situation happens for $C^*$-algebras
 \cite{Stinespring}, or even for locally $C^*$-algebras \cite{Inoue}, 
but it may fail even for pre $C^*$-algebras, in general. 
\end{remarks}

\subsection{Linear Maps with Values Adjointable Operators on VE-Modules.}
\label{ss:lmvaovem}
Given an ordered $*$-algebra $\cA$ and a VE-module $\cE$ over $\cA$, 
an \emph{$\cE$-reproducing kernel VE-module over $\cA$} is just an
$\cE$-reproducing kernel VE-space over $\cA$, with definition as in 
Subsection~\ref{ss:kvao}, which is also a VE-module over $\cA$. 
 
 \begin{proposition}\label{p:vhinvkolmomodule} 
 Let $\Gamma$ be a $*$-semigroup that acts 
on the nonempty set $X$ and let $\fk\colon X\times X\ra\cL^*(\cH)$ 
be a kernel, for some VE-module $\cH$ over an ordered $*$-algebra $\cA$. 
The following assertions are equivalent:

\begin{itemize}
\item[(1)] $\fk$ is positive semidefinite, in the sense of \eqref{e:npos}, 
and invariant under the action of $\Gamma$ on $X$, that is, 
\eqref{e:invariant} holds.
\item[(2)] $\fk$ has a $\Gamma$-invariant VE-module (over $\cA$) 
linearisation $(\cK;\pi;V)$.
\item[(3)] $\fk$ admits an $\cH$-reproducing kernel VE-module $\cR$ and 
there exists a $*$-representation $\rho\colon \Gamma\ra\cL^*(\cR)$ such that 
$\rho(\xi)\fk_xh=\fk_{\xi\cdot x}h$ for all $\xi\in\Gamma$, $x\in X$, $h\in\cH$.
\end{itemize}

In addition, in case any of the assertions \emph{(1)}, \emph{(2)}, or
\emph{(3)} holds,  
then a minimal $\Gamma$-invariant VE-module 
linearisation can be constructed,
any minimal $\Gamma$-invariant VE-module 
linearisation is unique up to unitary equivalence, 
a pair $(\cR;\rho)$ as in assertion \emph{(3)} with $\cR$ 
minimal can be always obtained and, in this case, it is uniquely 
determined by $\fk$ as well.
\end{proposition}

 \begin{proof} We use the notation as in the proof of 
Theorem~\ref{t:vhinvkolmo}. We actually prove only the implication 
(1)$\Ra$(2) since, as observed in Remark~\ref{r:vhinvkolmo}, that 
construction provides a $\Gamma$-invariant reproducing kernel VE-space
linearisation, while the other implications are not much different.
 
 (1)$\Ra$(2). We first observe that, 
 since $\cH$ is a module over $\cA$,
 the space $\cF(X;\cH)$ has a natural structure of right module over 
 $\cA$, more precisely, for any $f\in\cF(X;\cH)$ and $a\in \cA$
 \begin{equation*} (fa)(x)= f(x)a,\quad x\in X.
 \end{equation*}
 In particular, the space $\cF_0(X;\cH)$ is a submodule of $\cF(X;\cH)$. On
 the other hand, by assumption, for each $x,y\in X$, $\fk(x,y)\in\cL^*(\cH)$, 
 hence $\fk(x,y)$ is a module map. These imply that the convolution operator
 $K\colon\cF_0(X;\cH)\ra\cF(X;\cH)$ defined as in \eqref{e:convop} is a 
module map. Indeed, for any $f\in\cF_0(X;\cH)$, $a\in\cA$, and $y\in X$,
\begin{equation*} ((Kf)a)(x)=\sum_{x\in X}\fk(y,x)f(x)a=K(fa)(x).
\end{equation*}
Then, the space $\cG(X;\cH)$ which, with the definition as in 
\eqref{e:fezero}, is the range of the convolution operator $K$, 
is a module over $\cA$ as well. 

We show that, 
when endowed with the $\cA$ valued gramian $[\cdot,\cdot]_\cG$ defined
as in \eqref{e:ipfezero}, we have
\begin{equation}\label{e:gmp}
[e,fa]_\cG=[e,f]_\cG\, a,\quad e,f\in\cG(X;\cH),\ a\in\cA.
\end{equation}
To see this, let $e=Kg$ and $f=Kh$ for some $g,h\in\cF_0(X;\cH)$. Then,
\begin{equation*}[e,fa]_\cG=[Kg,ha]_{\cF_0}=\sum_{y\in X}
[e(y),h(y)a]_\cH=\sum_{y\in X}[e(y),h(y)]_\cH a=[Kg,h]_{\cF_0}a=[e,f]_\cG a.
\end{equation*}

From \eqref{e:gmp} and the proof of the implication (1)$\Ra$(2) in 
Theorem~\ref{t:vhinvkolmo}, it follows that $\cK=\cG(X;\cH)$ is a 
VE-module over the ordered $*$-algebra $\cA$ and hence, the triple
$(\cK;\pi;V)$ is a minimal $\Gamma$-invariant VE-module linearisation of 
$\fk$.
 \end{proof}
 
\begin{corollary}
\label{c:stinespringmodule}
Let $\phi\colon \cB\ra\cL^*(\cH)$ be a linear map, for some $*$-algebra $\cB$ 
and some VE-module $\cH$ over an ordered $*$-algebra $\cA$. 
The following assertions are equivalent:
\begin{itemize}
\item[(1)] The map $\phi$ is positive semidefinite.
\item[(2)] There exist a VE-module $\cK$ over the ordered $*$-algebra $\cA$, 
a linear map $V\colon\cB\ra\cL^*(\cH,\cK)$, and a $*$-representation 
$\pi\colon \cB\ra\cL^*(\cK)$, such that:
\begin{itemize}
\item[(i)] $\phi(a^*b)=V(a)^*V(b)$ for all $a,b\in\cB$.
\item[(ii)] $V(ab)=\pi(a)V(b)$ for all $a,b\in\cB$.
\end{itemize}
\end{itemize}
In addition, if this happens, then the triple $(\cK;\pi;V)$ 
can always be chosen 
minimal, in the sense that $\cK$ is the linear span of the set 
$V(\cB)\cH$, and any two minimal triples as before are unique, modulo 
unitary equivalence.
\begin{itemize}
\item[(3)] There exist an $\cH$-reproducing kernel VE-module $\cR$ on 
$\cA$ and a 
$*$-representation $\rho\colon\cB\ra\cL^*(\cR)$ such that:
\begin{itemize}
\item[(i)] $\cR$ has the reproducing kernel $\cB\times\cB
\ni(a,b)\mapsto \phi(a^*b)\in\cL^*(\cH)$.
\item[(ii)] $\rho(a)\phi(\cdot b)h=\phi(\cdot ab)h$ for all 
$a,b\in\cB$ and $h\in\cH$.
\end{itemize}
\end{itemize}
In addition, the reproducing kernel VE-module $\cR$ as in (3) can be 
always constructed minimal and in this case it is uniquely determined by 
$\phi$.
\end{corollary}

In case the $*$-algebra $\cB$ is unital, Corollary~\ref{c:stinespringmodule} 
takes a form that reveals the fact that it is actually a non topological 
version of 
Kasparov's Theorem \cite{Kasparov} and its generalisation \cite{Joita}.

\begin{corollary}\label{c:stinespringmodule2} 
Let $\cB$ be a unital $*$-algebra and 
$\phi\colon\cA\ra\cL^*(\cH)$ a linear map, for 
some VE-module $\cH$ over an ordered $*$-algebra $\cA$. 
The following assertions are equivalent:
\begin{itemize}
\item[(i)] $\phi$ is positive semidefinite.
\item[(ii)] There exist a VE-module $\cK$ over $\cA$, a $*$-representation
 $\pi\colon\cB\ra\cL^*(\cK)$, and $W\in\cL^*(\cH,\cK)$ such that 
 \begin{equation}\phi(b)=W^*\pi(b)W,\quad b\in\cB.\end{equation}
\end{itemize}
In addition, if this happens, then the triple $(\cK;\pi;W)$ can always be chosen 
minimal, in the sense that $\cK$ is the linear span of the set 
$\pi(\cA)W\cH$, and any two minimal triples as before are unique, 
modulo unitary equivalence.
\end{corollary}


\begin{thebibliography}{99}

\bibitem{Aronszajn} \textsc{N. Aronszajn}, Theory of reproducing kernels, 
\textit{Trans. Amer. Math. Soc.} \textbf{68}(1950), 337-404.

\bibitem{Arveson} \textsc{W.B. Arveson}, \textit{A Short Course on Spectral
Theory}, Graduate Texts in Mathematics, Vol. 209, Springer-Verlag, 
Berlin -- Heidelberg -- New York 2002.

\bibitem{ConstantinescuGheondea1} \textsc{T.~Constantinescu, 
A.~Gheondea}: Representations of hermitian kernels by means of Krein 
spaces, {\it Publ. RIMS. Kyoto Univ.}, {\bf 33}(1997), 917--951.

\bibitem{ConstantinescuGheondea2} {\sc T. Constantinescu, A. Gheondea}, 
Representations of Hermitian kernels by means of Krein spaces. II. Invariant 
kernels, \textit{Comm. Math. Phys.} \textbf{216}(2001), 409Ð430.

\bibitem{EsslamzadehTaleghani} \textsc{G.H. Esslamzadeh, F. Taleghani},
Structure of quasi operator systems, \textit{Linear Algebra
Appl.}, \textbf{438}(2013), 1327--1392.

\bibitem{EsslamzadehTurowska} \textsc{G.H. Esslamzadeh, L. Turowska},
Arveson's Extension Theorem in $*$-algebras,  arXiv:1311.5065

\bibitem{EvansLewis} \textsc{D.E.~Evans, J.T.~Lewis}: \textit{Dilations of 
Irreducible Evolutions in Algebraic Quantum Theory}, {Comm. Dublin Inst. 
Adv. Studies Ser. A} No. 24, Dublin Institute for Advanced Studies, Dublin, 
1977.

\bibitem{GasparGaspar} \textsc{D. Ga\c spar, P. Ga\c spar},
An operatorial model for Hilbert 
$\cB(X)$-modules, \textit{Analele Univ. de Vest Timi\c soara} Ser. Mat.-Inform. 
\textbf{40}(2002), 15--29.

\bibitem{GasparGaspar2} \textsc{D. Ga\c spar, P. Ga\c spar}, Reproducing kernel 
Hilbert $B(X)$-modules, \textit{An. Univ. Vest Timi\c soara Ser. Mat.-Inform.} 
\textbf{43}(2005), no. 2, 47--71.

\bibitem{GasparGaspar1} \textsc{D. Ga\c spar, P. Ga\c spar}, 
Reproducing kernel Hilbert modules over locally $C^*$-algebras, 
\textit{An. Univ. 
Vest Timi\c soara Ser. Mat.-Inform.} \textbf{45}(2007), no. 1, 245--252.

\bibitem{Gheondea} \textsc{A. Gheondea}, Dilations of some VH-spaces 
operator valued kernels, \textit{Integral Equations and Operator Theory}, 
\textbf{74}(2012), 451--479.

\bibitem{Heo} \textsc{ J. Heo}, Reproducing kernel Hilbert $C^*$-modules and
kernels associated to cocycles, \textit{J. Mathematical Physics}, \textbf{49}(2008),
103507.

\bibitem{Inoue} \textsc{A. Inoue}, Locally $C^*$-algebras, 
\textit{Mem. Fac. Sci. Kyushu Univ. Ser. A}, \textbf{25}(1971), 197--235.

\bibitem{Joita} \textsc{M. Joi\c ta}, Maps between 
locally $C^*$-algebras, \textit{J. London Mat. Soc.}, \textbf{2}(66)(2002), 421--432.

\bibitem{Kasparov} \textsc{G.G. Kasparov}, Hilbert $C^*$-modules: theorems 
of Stinespring and Voiculescu, \textit{J. Operator Theory}, \textbf{4}(1980), 
133--150.

\bibitem{Lance} \textsc{E.C. Lance}, \textit{Hilbert $C\sp *$-Modules.
A toolkit for operator algebraists},
London Mathematical Society Lecture Note Series, 210. Cambridge 
University Press, Cambridge 1995.

\bibitem{Loynes1} \textsc{R.M. Loynes}, On generalized positive-definite
  functions, \textit{Proc. London Math. Soc. III. Ser.} \textbf{15}(1965),
  373--384.

\bibitem{Loynes2} \textsc{R.M. Loynes}, Linear operators in $VH$-spaces,
  \textit{Trans. Amer. Math. Soc.}  \textbf{116}(1965), 167--180.

\bibitem{Loynes3} \textsc{R.M. Loynes},
Some problems arising from spectral analysis, in
\textit{Symposium on Probability
Methods in Analysis (Loutraki, 1966)},  pp. 197--207, Springer, Berlin 1967.

\bibitem{ManuilovTroitsky} \textsc{V. Manuilov, E. Troitsky}, 
\textit{Hilbert $C^*$-
Modules}, Amer. Math. Soc., Providence R.I. 2005.

\bibitem{Murphy} \textsc{G.J. Murphy}, Positive definite kernels and Hilbert 
$C^*$-modules, \textit{Proc. Edinburgh Math. Soc.} \textbf{40}(1997), 367--374.

\bibitem{Naimark1} \textsc{M.A. Naimark}, On the representations of additive 
operator set functions, \textit{C.R. (Doklady) Acad. Sci. USSR} 
\textbf{41}(1943), 359--361.

\bibitem{Naimark2} \textsc{M.A. Naimark}, Positive definite operator functions 
on a commutative group, \textit{Bull. (Izv.) Acad. Sci. USSR} \textbf{7}(1943), 
237--244.

\bibitem{ParthasarathySchmidt} {\sc K.R.~Parthasaraty, K.~Schmidt}: 
{\it Positive-Definite Kernels, Continous Tensor Products and Central Limit 
Theorems of Probability Theory}, Lecture Notes in Mathematics, Vol. {\bf
272}, Springer-Verlag, Berlin 1972.

\bibitem{Paulsen} \textsc{V.R. Paulsen}, \textit{Completely Bounded Maps 
and Operator Algebras}, Cambridge University Press, Cambridge 2002.

\bibitem{PaulsenTodorovTomforde} \textsc{V.R. Paulsen, I.G. Todorov, M. 
Tomforde}, Operator system structures on ordered spaces, \textit{Proc. London Math. Soc.}, \textbf{102}(2011), 25--49.

\bibitem{PaulsenTomforde} \textsc{V.R. Paulsen, M. Tomforde}, 
Vector spaces with an order unit, \emph{Indiana Univ. Math. J.}, \textbf{58}(2009), 1319--1359.

\bibitem{Phillips} \textsc{N.C. Phillips}, Inverse limits of $C^*$-algebras, 
\textit{J. Operator Theory}, \textbf{19}(1988), 159--195.

\bibitem{Skeide} \textsc{M. Skeide}, \textit{Hilbert Modules and Applications
in Quantum Probability}, Habilitationschrift, Cottbus, 2001, 
http://www.math.tu-cottbus.de/INSTITUT/lswas/$\underline{\phantom{a}}$skeide.html

\bibitem{Saworotnow} \textsc{P.P. Saworotnow},
Linear spaces with an $H^*$-algebra-valued inner product,
\textit{Trans. Amer. Math. Soc.} (2)\textbf{262}(1980), 543--549.

\bibitem{Stinespring} \textsc{W.F. Stinespring},
Positive functions on $C\sp *$-algebras,
\textit{Proc. Amer. Math. Soc.} \textbf{6}(1955), 211--216.

\bibitem{Szafraniec} \textsc{F.H. Szafraniec}, Murphy's positive definite 
kernels and Hilbert $C^*$-modules reorganized, in \textit{Banach Center
  Publications}, Vol. 89, pp. 275--295, Warszawa 2010.

\bibitem{BSzNagy} \textsc{B. Sz.-Nagy}, Prolongement des transformations 
de l'espace de Hilbert qui sortent de cet espace, in \textit{
Appendice au livre ``Le\c cons d'analyse fonctionnelle''} par F. Riesz et
B. Sz.-Nagy, pp. 439-573  Akademiai Kiado, Budapest, 1955.

\bibitem{Weron} \textsc{A. Weron}, Prediction theory in Banach spaces. In: 
\textit{Proceedings of the Winter School of Probability, Karpacz}, 
pp. 207--228. Springer Verlag, Berlin 1975.
\end{thebibliography}
\end{document}